\documentclass{elsarticle}

\usepackage{graphicx} 
\usepackage{xcolor}
\usepackage{tikz}
\usepackage{amsmath}
\usepackage{mathtools}

\usepackage[active]{srcltx}
\usepackage{amssymb}
\usepackage{amsthm}
\usepackage{latexsym}
\usepackage{color}
\usepackage{dsfont}
\usepackage[normalem]{ulem}
\usepackage{epsfig}
\usepackage{verbatim}
\usepackage{enumerate}
\usepackage{float}
\usepackage{url}

\numberwithin{equation}{section}

\newtheorem{theorem}{Theorem}[section]

\newtheorem{lemma}[theorem]{Lemma}
\newtheorem{proposition}[theorem]{Proposition}
\newtheorem{definition}[theorem]{Definition}
\newtheorem{corollary}[theorem]{Corollary}
\newtheorem{remark}[theorem]{Remark}

\begin{document}

\begin{frontmatter}

\title{A method of moving planes for a mixed local-nonlocal parabolic equation with quenching}

\author[1]{Sergio Junquera\corref{cor}}

\ead{sejunque@ucm.es}

\affiliation[1]{organization={Departamento de Análisis Matemático y Matemática Aplicada, Universidad Complutense de Madrid},
			addressline={Plaza de Ciencias 3},
            city={Madrid},
            postcode={28040}, 
            state={Madrid},
            country={Spain}}

\author[2,3]{Leandro M. Del Pezzo}

\ead{leandro.delpezzo@fcea.edu.uy}

\affiliation[2]{organization={Instituto de Estadística (IESTA), Departamento de Métodos Cuantitativos, Facultad de Ciencias Económicas y de Administración, Universidad de la República},
            addressline={Av. Gonzalo Ramírez 1926}, 
            city={Montevideo},
            postcode={11200}, 
            state={Montevideo},
            country={Uruguay}}

\affiliation[3]{organization={PEDECIBA Matemática},
            addressline={Iguá 4225}, 
            city={Montevideo},
            postcode={11400}, 
            state={Montevideo},
            country={Uruguay}}

\cortext[cor]{Corresponding author}

\begin{abstract}
In this paper we develop a parabolic version of the method of moving planes on systems with a mixed local and nonlocal diffusion operator, that being $\mathcal{L} = -(-\Delta)^s + \Delta$, where $(-\Delta)^s$ is the integral fractional laplacian. We then use this result to study the quenching phenomena occurring in said systems with a singular absorption term of the type $(1-u)^{-p}$. We prove that there exists a range of multiplicative parameters for which every solution presents quenching, and that for any initial datum $u_0$, there exists a small enough multiplicative parameter such as the solution is globally defined. We also prove that, under appropriate initial data, the quenching phenomena cannot happen arbitrarily close to the boundary, but they are contained in a compact subset of the domain.
\end{abstract}


\end{frontmatter}

\section{Introduction}

In this paper we will study the following nonlinear nonlocal parabolic dynamical system with Dirichlet boundary conditions:

 \begin{equation} \label{1.1}
    \left\{
    \begin{array}{ll}
        \displaystyle u_t(x,t)= \mathcal{L} u(x,t)+\lambda (1-u)^{-p}(x,t), & (x,t)\in\Omega\times(0,T), \\
        u(x,0)=u_0(x), & x\in\Omega,\\
        u(x,t) = 0, & (x,t)\in (\mathbb{R}^n \setminus \Omega) \times [0,T).
    \end{array}
    \right.
    \end{equation}
    We consider that $T\in(0,\infty]$ is the maximal existence time of the solution, $\Omega \subset \mathbb{R}^n$ is an open bounded connected domain that is $C^{2,\gamma}$, the initial data $u_0$ is a function in $H_0^1(\Omega) \cap L^\infty (\Omega)$ such that $\|u_0\|_\infty <1$ and $u_0(x)\geq 0$ for every $x\in\Omega$, and the parameters $\lambda,p>0$. The diffusion operator $\mathcal{L}$ is a combination of a nonlocal operator (the fractional laplacian) and a local operator (the classical laplacian),
    \begin{equation*}
    	\mathcal{L} = -(-\Delta)^s + \Delta,
    \end{equation*}
    with $s\in(0,1)$ and
    \begin{equation} \label{eq.fraclap}
    	(-\Delta)^s \psi(x) = C_{n,s} \text{P.V.} \int_{\mathbb{R}^n} \frac{\psi(x)-\psi(y)}{|x-y|^{n+2s}} dy,
    \end{equation}
    where $C_{n,s}$ is a normalization constant that depends on $n$ and $s$ and P.V. denotes the principal value integral. This principal value is used to avoid the singularity of the integrand at the diagonal $x=y$, and is defined as
	\begin{equation*}
		\text{P.V.} \int_{\mathbb{R}^n} F(x,y) dy = \lim_{\varepsilon\to 0^+} \int_{\mathbb{R}^n \setminus B_\varepsilon (x)} F(x,y) dy.
	\end{equation*}   
     This kind of operators generate high mathematical interest due to the combination of local and nonlocal mechanisms, as well as due to the lack of scale invariance, see for example \cite{AC2021,BDVV2023,SVWZ2025}. They are also highly interesting operators from an application viewpoint, being able to model diffusion phenomena with different time scales and bimodal power-law distribution processes, see \cite{PV2021}, among others.
    
    In this paper we will work with Dirichlet boundary conditions, which must be imposed in the whole complement of the domain, $\mathbb{R}^n \setminus \Omega$. This is natural in the field of nonlocal operators, since their nature forces us to establish the value of the solution in the whole space $\mathbb{R}^n$, see \cite{AMRT2010, BFRW1997, BV2016, F2003}. In our case, this is needed so that expression \eqref{eq.fraclap} makes sense.
    
    Notice that the equation in our system has a singular absorption term. Therefore, the solutions should increase at some point of $\Omega$ and they could eventually reach the value $1$ there. If this happens at a finite time $t=T_1$, the absorption term blows up and the classical solution no longer exists, then $T=T_1$. We say in this situation that quenching happens and $T$ is the quenching time. More precisely, we say that a solution of system \eqref{1.1} given by $u$ presents quenching in finite time $T$ if
    \begin{equation*}
    	\limsup_{t\nearrow T} \|u(\cdot,t)\|_{L^\infty(\Omega)} = 1.
    \end{equation*}
    We will often refer to a solution that presents quenching as a quenching solution or just say that the solution quenches. If we work with solutions of \eqref{1.1} that present enough regularity, we can also define their quenching set as
    \begin{equation*}
    	Q(u) = \{x\in\overline\Omega : \exists \, t_n \nearrow T, x_n \to x \text{ such that } u(x_n,t_n) \nearrow 1\},
    \end{equation*}
    which contains the points in space in which the phenomenon of quenching happens.
    
    The phenomenon of quenching appears naturally in physical models such as the nonlinear heat conduction in solid hydrogen, see \cite{R1979}, or the Arrhenius Law in combustion theory, see \cite{CK1995}. Quenching was studied for the first time in \cite{K1974} for the problem
    \begin{equation*}
    	v_t = v_{xx} + (1-v)^{-1},
    \end{equation*}
    where quenching happens when $v$ reaches the value $v=1$, the same as in our system. Since then, the phenomenon of quenching for different problems has been the issue of intensive study for local diffusion operators, see for example the surveys \cite{C1996,FL1993,L1985,L1989,L1993} for a single equation, or \cite{FdQR2006,JZZ2013,dQR2002,ZW2008} for systems.
    
    The problem of quenching with nonlocal diffusion has been less studied. In \cite{F2011,ZMZ2011}, the authors study the system with a single equation and a nonlocal diffusion operator with a smooth kernel under Dirichlet and Neumann conditions respectively. In a similar fashion, \cite{AFJ2025,J2025} extend the results to a system of two equations with coupled absorption terms under the same nonlocal diffusion operator, studying additional matters such as the difference between simultaneous and nonsimultaneous quenching.
    
    To study the phenomenon of quenching, some of the most important techniques rely heavily on the maximum principle and comparison results. While in the local framework these need hypothesis over the boundary of our domain $\Omega$, it is well known that as soon as we move onto a nonlocal context, these hypothesis turn much more restrictive, needing assumptions over the whole $\mathbb{R}^n \setminus \Omega$, see for example \cite{AMRT2010}.
    
    These assumptions may take the form of, for example, asking the solution to be positive in the whole complement $\mathbb{R}^n \setminus \Omega$. This is sometimes too restrictive, and needs to be somewhat relaxed. One way to do so is to take advantage of the symmetry invariance of our operator, in a way that is traditionally called the method of moving planes.
    
    The method of moving planes has been traditionally studied in the elliptic local case in papers such as \cite{GNN1979, S1971}, and more recently these results have been adapted to the elliptic nonlocal case, see \cite{CLL2016}.
    
    The aim of this article is twofold. On the one hand, we are interested in extending the method of moving planes to the parabolic case with a nonlocal diffusion operator of the form $\mathcal{L}$. On the other hand, we strive to study properties of the solutions of system \eqref{1.1}, such as the presence of quenching in finite time, globally defined solutions, and quenching at points of the boundary.

    \subsection{Main results}
    
    Our first objective is to extend the method of moving planes presented in \cite{CLL2016} for a parabolic system involving our mixed operator $\mathcal{L}$, which will be necessary to prove later results for which the classical maximum principle is too restrictive.
    
    \begin{theorem} \label{teo.movingplanes}
	Let $\Psi$ be a hyperplane in $\mathbb{R}^n$, $H$ one of the two half-spaces determined by said hyperplane, let $\hat{H}$ be the reflection of $H$ and, for any $x\in H$, let $\hat{x}$ be the reflection of $x$ with respect to $\Psi$. Let $\tau \geq 0$, $D$ be an open bounded domain symmetric with respect to $\Psi$ and define $D^- = D \cap H$, $D^+ = D \cap \hat{H}$. 
	
	Suppose that $u$ is a function in $\mathcal{C}([\tau,T), L^1_{2s} (\mathbb{R}^n))$ such that $u(\cdot,t) \in C^{2,\alpha}_{loc}(D)$ for every $t\in[\tau,T)$, $u_t \in C^\beta_{loc} (D^- \times[\tau,T))$ and $u(\cdot,t)$ is lower semi-continuous on $\overline{D} \times [\tau,T^\prime]$ for every $T^\prime \in (\tau,T)$. Assume that $u$ satisfies
    \begin{equation} \label{cond.movingplanes}
    \left\{
    \begin{array}{ll}
        \displaystyle u_t(x,t) - \mathcal{L} u(x,t)-c(x,t)u(x,t)\geq 0, & (x,t)\in D^-\times(\tau,T), \\
        u(x,\tau)\geq 0, & x\in D^-, \\
        u(x,t)\geq 0, & (x,t)\in (H\setminus D^-)\times[\tau,T),\\
        u(x,t) = -u(\hat{x},t), & (x,t)\in H \times [\tau,T),
    \end{array}
    \right.
    \end{equation}
    where there exists some $C, c_0 \in L^\infty_{loc}([\tau,T))$ such that, for every $t\in[\tau,T)$, 
    $c_0(t)\le c(x,t)\leq C(t)$ uniformly in $D^-$. Then we have that $u(x,t) \geq 0$ for every $(x,t)\in D^- \times[\tau,T)$.
    
	Furthermore, assume additionally that $D$ has a $C^{1,1}$ boundary, $u \in C((\tau,T), H^1(\mathbb{R}^n))$ and $u_t \in L^1((\tau,T), L^1(D))$. Then, if $u(x_1,t_1) = 0$ at some point $(x_1,t_1) \in D^- \times (\tau,T)$, then $u\equiv 0$ in $D^- \times[\tau,t_1]$.
\end{theorem}

	Our second objective is to obtain various properties on the behaviour of the quenching phenomena that solutions of \eqref{1.1} present. First, we prove that there exists a threshold $\lambda^\ast(p,\Omega)>0$ such that, if $\lambda>\lambda^\ast$, then every solution of \eqref{1.1} presents quenching in finite time. What is more, the threshold $\lambda^\ast$ becomes bigger as the domain shrinks.
	
	\begin{proposition} \label{prop.todoquenchea}
    Let $\mu_1$ be the principal Dirichlet eigenvalue of $-\mathcal{L}$ as defined in \eqref{eq.principaleigenvalue}. If $\lambda > \mu_1 (1+p)^{-1}$, then every solution of \eqref{1.1} presents quenching in finite time.
	\end{proposition}
	
	We also prove that, for any initial data $u_0$, one can take a sufficiently small $\lambda$ such that the solution of \eqref{1.1} is global in time. This time, the behaviour of the threshold differs significantly with the size of the domain: it inherits the fractional laplacian order when the domain is large, and the classical laplacian order if the domain is small.
	
	\begin{proposition} \label{prop.solsglobales}
    Let $\delta\in(0,1)$ and $u_0$ be some initial data such that $\|u_0\|_\infty \leq 1-\delta$. Let $x_0\in\Omega$ and $R>0$ such that $\overline\Omega \subset B_R(x_0)$. Define $\lambda_0>0$ as
    \begin{equation*}
        \lambda_0 = \left(\frac{\delta}{2}\right)^{1+p} \max \left\{ \frac{1}{A_{n,s} R^{2s}}, \frac{2n r^2}{R^2} \right\},
    \end{equation*}
    where $A_{n,s} = \frac{\Gamma(n/2)}{2^{2s}\Gamma(s+n/2)\Gamma(s+1)}$ and $r\in(0,1)$ is fixed and independent of $x_0$, $\Omega$, $R$.
    
    If $\lambda \leq \lambda_0$, the solution $u$ of \eqref{1.1} with initial data $u_0$ and parameter $\lambda$ is global.
\end{proposition}

	Finally, we prove that the quenching set of any solution of \eqref{1.1} that presents quenching in finite time lies in a compact set inside of $\Omega$. Particularly, no solution of \eqref{1.1} with smooth initial data can present quenching in points arbitrarily close to $\partial\Omega$.
	
	\begin{theorem} \label{teo.quenchingset}
	Let $u_0 \in C^{1}(\overline\Omega)$ such that $\partial_n u_0(x) <0$ for every $x\in\partial\Omega$, let $\Omega$ be convex and $u$ the solution to \eqref{1.1} with initial data $u_0$. Then there exists some $\kappa>0$ such that $Q(u) \subset \Omega \setminus \Omega_{\kappa}$, where $\Omega_{\kappa} = \{x\in\Omega : dist(x,\partial\Omega) < \kappa\}$.
\end{theorem}

We note that all these results are also valid in the local case where the operator is just $\Delta$: the method of moving planes reduces to the local maximum principle, and Proposition \ref{prop.todoquenchea}, Proposition \ref{prop.solsglobales} and Theorem \ref{teo.quenchingset} have already been proven for the classical laplacian (see for example \cite{ZW2008}), with our thresholds being compatible to those that were already known.

More interestingly, most of our proofs also extend naturally to the fully fractional case with operator $-(-\Delta)^s$, which, as far as the authors know, had not been studied yet. Many of the arguments adapt naturally to this case as long as solutions are well-defined and classical.

\medskip

\noindent\textbf{The paper is structured as follows:} in Section \ref{sec.preliminary} we will introduce some relevant definitions, notation and preliminary results. In particular, we will state the existence and uniqueness of solutions, as well as their regularity, see Proposition \ref{prop.existence}. Then, in Section \ref{sec.positivity}, we will obtain a maximum principle in weak formulation, see Lemma \ref{lema_maximumweak}, that will allow us to prove continuity of solutions up to the boundary and, after that, we state a pointwise comparison result, see Lemma \ref{lemma_comparison_particularized}. We will also prove the existence and regularity of odd eigenfunctions on symmetric domains, see Lemma \ref{lemma.oddeigenf}. This result is interesting by itself, and it will be very useful for other purposes in conjunction with the parabolic moving planes. In Section \ref{sec.movingplanes}, we prove Theorem \ref{teo.movingplanes}, that is, the parabolic moving planes result for the general class of problems \eqref{cond.movingplanes}, which is one of our main objectives. In Section \ref{sec.quenching}, we will prove the quenching phenomena properties. First, we will obtain sufficient conditions over $\lambda$ under which every solution of \eqref{1.1} presents quenching in finite time, see Proposition \ref{prop.todoquenchea}. We will also obtain, in Proposition \ref{prop.solsglobales}, sufficient conditions on $\lambda$ depending on the initial datum $u_0$ such that the solution to \eqref{1.1} is globally defined. Finally, we will prove Theorem \ref{teo.quenchingset}, in which we determine that quenching cannot happen at the boundary of our domain when the initial data are regular enough.

\section{Preliminary notions and results} \label{sec.preliminary}

	In this section we will properly state definitions and results that will be used throughout the paper, taking special care with the regularity of our operator $\mathcal{L} = -(-\Delta)^s + \Delta$. Throughout this section, we will assume that $\Omega$ is an open bounded connected $C^{2,\gamma}$ domain, $k \in \mathbb{N}$, $\alpha \in (0,1)$ and $s\in(0,1)$.
	
	\subsection{Notation and spaces}
	
	First, let us define some of the spaces that will be used extensively in our work.
	
	\begin{align*}
		H_0^1(\Omega) =& \left\{\varphi \in L^2(\Omega) : \nabla \varphi \in L^2(\Omega) \; \text{ and } \; \text{Tr }\varphi = 0 \text{ on } \partial\Omega \right\}, \\
		D(\mathcal{L}) =& \left\{\varphi \in L^2(\Omega) : \mathcal{L} \varphi \in L^2(\Omega) \right\}, \\
		C^k(\Omega) =& \left\{\varphi: \Omega \to \mathbb{R} : D^\beta \varphi \in C(\Omega) \text{ for every } |\beta| \leq k \right\},\\
		C^{k,\alpha} (\Omega) =& \left\{\varphi \in C^k(\Omega) : D^\beta \varphi \in C^{\alpha}(\Omega) \text{ for every } |\beta|=k \right\},\\
		C^{k,\alpha}_{loc} (\Omega) =& \left\{\varphi \in C^k(\Omega) : D^\beta \varphi \in C^{\alpha}_{loc}(\Omega) \text{ for every } |\beta|=k \right\}, \\
		C^{k,1} (\Omega) =& \left\{\varphi \in C^k(\Omega) : D^\beta \varphi \in \text{Lip}(\Omega) \text{ for every } |\beta|=k \right\},\\
		C^{k,1}_{loc} (\Omega) =& \left\{\varphi \in C^k(\Omega) : D^\beta \varphi \in \text{Lip}_{loc}(\Omega) \text{ for every } |\beta|=k \right\},\\
		L^1_{2s} (\mathbb{R}^n) =& \left\{ \varphi \in L^1_{loc}(\mathbb{R}^n) : \int_{\mathbb{R}^n} \frac{|\varphi(y)|}{1+|y|^{n+2s}} dy <\infty \right\}.
	\end{align*}
	
	Notice that the space $H_0^1(\Omega)$ can be naturally identified with the space
	\begin{equation*}
		\tilde{H}_0^1(\Omega) = \left\{\varphi \in L^2(\mathbb{R}^n) : \nabla \varphi \in L^2(\mathbb{R}^n) \; \text{ and } \; \varphi \equiv 0 \text{ a.e. in } \mathbb{R}^n\setminus \Omega  \right\},
	\end{equation*}
	which is the actual space we need to use due to the nonlocal nature of our operator $\mathcal{L}$, see \cite{BDVV2023}. However, we will identify these two spaces implicitly throughout the paper and every $H_0^1(\Omega)$ function with its extension on $\tilde{H}_0^1(\Omega)$.
	
	The space $L^1_{2s}(\mathbb{R}^n)$ is often called the ``tail space'', and it contains the functions whose ``decay at infinity'' is controlled enough for $(-\Delta)^s \varphi$ to be well defined. Indeed, it is well known that $(-\Delta)^s \varphi$ is well defined pointwise if $\varphi \in C^{1,1}_{loc}(\Omega) \cap L^1_{2s}(\mathbb{R}^n)$, see \cite{S2007}.
	
	\subsection{Existence, uniqueness and regularity of solutions}
	
	Let us now define the concept of classical solution of \eqref{1.1} as follows.
	
	\begin{definition} \label{def.solution}
		We say that a function
		\begin{equation*}
			u \in C([0,T),H_0^1(\Omega)) \cap C((0,T), D(\mathcal{L})) \cap C^1((0,T), L^2(\Omega)),
		\end{equation*}
		such that $u,u_t \in C^\alpha_{loc} (\Omega \times (0,T))$ and $u(\cdot,t) \in C^{2,\alpha}_{loc}(\Omega)$ for every $t\in(0,T)$ and some $\alpha \in (0,1)$, is a classical solution of equation \eqref{1.1} if it satisfies the equalities in \eqref{1.1} pointwise for every $(x,t) \in \Omega \times (0,T)$.
	\end{definition}
	
	Based on this definition, we can state a result on existence, uniqueness and regularity of solutions of \eqref{1.1}. This was proved in \cite{JP2026} by the authors.
	
	\begin{proposition} \label{prop.existence}
		The equation \eqref{1.1} admits a unique classical solution $u$. If $T$ is its maximal definition time and $T<\infty$, then $\limsup_{t\to T} \|u(\cdot,t)\|_\infty = 1$.
	\end{proposition}
	
	This proposition gives us higher interior regularity that allows us to work with solutions for which $\Delta u(\cdot,t)$ and $(-\Delta)^s u(\cdot,t)$ are well defined pointwise for every $t\in(0,T)$.
	
	\subsection{The eigenvalue problem}
	
	Another important ingredient we will need is the existence of eigenfunctions of our problem. First, define the Dirichlet form associated to our operator $\mathcal{L}$, which is:
	\begin{equation} \label{eq.dirichletform}
		\mathcal{E}(\varphi,\psi) = \frac{C_{n,s}}{2}\iint_{\mathbb{R}^n\times\mathbb{R}^n} \frac{(\varphi(x)-\varphi(y))(\psi(x)-\psi(y))}{|x-y|^{n+2s}} dy dx + \int_\Omega \nabla\varphi \nabla\psi dx,
	\end{equation}
	for every $\varphi,\psi \in H_0^1(\Omega)$.
	
	Then we can define $\mu_1$ as the principal Dirichlet eigenvalue of our operator $-\mathcal{L}$ as
\begin{equation} \label{eq.principaleigenvalue}
    \mu_{1} = \inf_{u\in H_0^1(\Omega)\setminus \{0\}} \frac{\mathcal{E}(u,u)}{\displaystyle \int_\Omega |u|^2 dx} \in [0,\infty).
\end{equation}
In \cite{BDVV2023}, they prove that there exists an associated eigenfunction $\varphi_1 \in \mathcal{C}^{1,\beta}(\overline\Omega)$ with $\int_\Omega\varphi_{1}=1$ that satisfies $\varphi_1\geq 0$ and is a classical solution of the problem
\begin{equation} \label{eq.eigensystem}
    \left\{
    \begin{array}{ll}
        \displaystyle -\mathcal{L} \varphi_{1}(x)= \mu_{1} \varphi_{1}(x), & x\in\Omega, \\
        \varphi_1 (x) = 0, & x \in\mathbb{R}^n \setminus \Omega.
    \end{array}
    \right.
\end{equation}
We refer to $\varphi_{1}$ as the principal eigenfunction of $-\mathcal{L}$ in $\Omega$.
	
	\section{Maximum principle, comparison results and odd eigenfunctions} \label{sec.positivity}
	
	In the study of quenching phenomena, one of the most useful techniques is the usage of maximum and comparison principles. As such, we will devote this section to proving these results, as well as the existence of symmetric eigenfunctions with which we will need to compare our solutions during the proofs of our main theorems.

	\subsection{Maximum principle and comparison results}
	
	Let us first establish a weak parabolic maximum principle that allows us to prove that every solution of \eqref{1.1} is continuous up to the boundary via comparison with a suitable barrier.
	
	\begin{lemma} \label{lema_maximumweak}
	Let $D \subset \mathbb{R}^n$ be an open bounded Lipschitz domain and $v\in L^2((0,\tau), H^1(\mathbb{R}^n)) \cap C([0,\tau], L^2(D))$ such that $v_t \in L^2((0,\tau), (H_0^1(D))^\prime)$. Let $a \in L^\infty((0,\tau),L^\infty(D))$. If
	\begin{equation} \label{eq_weakineq}
		\left\{ \begin{array}{ll}
			v_t(x,t) - \mathcal{L} v(x,t) +a(x,t)v(x,t) \geq 0, & (x,t) \in D \times (0,\tau), \\
			v(x,0) \geq 0, & x\in D, \\
			v(x,t) \geq 0, & (x,t) \in (\mathbb{R}^n \setminus D) \times [0,\tau],
		\end{array} \right.
	\end{equation}
	hold in the weak sense, then $v(x,t) \geq 0$ a.e. $(x,t) \in \mathbb{R}^n \times (0,\tau)$.
\end{lemma}

\begin{proof}
	The first weak inequality in \eqref{eq_weakineq} tells us that, for every nonnegative test function $\varphi \in H_0^1(D)$,
	\begin{align} \label{eq_weakineq2}
		\nonumber \int_D v_t(x,t) \varphi(x) dx + \int_D \nabla v(x,t)  \nabla \varphi(x) dx &+ \int_D a(x,t)v(x,t)\varphi(x,t) dx \\
		&+ \frac{C_{n,s}}{2} \iint_{\mathbb{R}^n \times \mathbb{R}^n} \frac{(v(x,t)-v(y,t))(\varphi(x)-\varphi(y))}{|x-y|^{n+2s}} dx dy \geq 0,
	\end{align}
	a.e. $t\in(0,\tau)$. Now set $v_- = \max\{-v,0\}$, which satisfies that $v_-(x,t) = 0$ a.e. $(x,t) \in (\mathbb{R}^n \setminus D) \times [0,\tau]$. This together with the fact that $v(\cdot,t) \in H^1(\mathbb{R}^n)$ a.e. $t\in(0,\tau)$ implies that $v_- (\cdot,t) \in H_0^1(D)$ a.e. $t\in(0,\tau)$. Then we can test with $v_-$ in \eqref{eq_weakineq2}, and get
	\begin{align*}
		-\frac{1}{2}\frac{d}{dt} \|v_-(\cdot,t)\|^2_{L^2(D)} &- \int_D |\nabla v_- (x,t)|^2 dx - \int_D a(x,t) v_-^2 (x,t) dx \\
		&+ \frac{C_{n,s}}{2} \iint_{\mathbb{R}^n \times \mathbb{R}^n} \frac{(v(x,t)-v(y,t))(v_-(x,t)-v_-(y,t))}{|x-y|^{n+2s}} dx dy \geq 0.
	\end{align*}
	a.e. $t\in(0,\tau)$. Then, using the fact that
	\begin{equation*}
		(v(x,t)-v(y,t))(v_-(x,t)-v_-(y,t)) \leq - |v_-(x,t) - v_-(y,t)|^2
	\end{equation*}
	for every $(x,t), (y,t) \in \mathbb{R}^n \times [0,\tau]$, we get
	\begin{align*}
	-\frac{1}{2}\frac{d}{dt} \|v_-(\cdot,t)\|^2_{L^2(D)} &- \int_D |\nabla v_- (x,t)|^2 dx - \int_D a(x,t) v_-^2(x,t) dx \\
	&- \frac{C_{n,s}}{2} \iint_{\mathbb{R}^n \times \mathbb{R}^n} \frac{|v_-(x,t)-v_-(y,t)|^2}{|x-y|^{n+2s}} dx dy \geq 0
	\end{align*}
	a.e. $t \in (0,\tau)$, which in turn implies
	\begin{equation*}
		\frac{d}{dt} \|v_-(\cdot,t)\|^2_{L^2(D)} \leq -2\int_D a(x,t) v_-^2(x,t) dx \leq 2A \|v_-(\cdot,t)\|^2_{L^2(D)},
	\end{equation*}
	a.e. $t \in (0,\tau)$ where $A= \|a\|_{L^\infty((0,\tau), L^\infty(D))}$. Finally, using Grönwall's inequality, we conclude
	\begin{equation*}
		0 \leq \|v_-(t)\|^2_{L^2(D)} \leq \|v_-(0)\|^2_{L^2(D)} e^{2At} = 0,
	\end{equation*}
	where we have used the fact that $v(x,0) \geq 0$ a.e. $x\in D$ for the last inequality. Then, $v_-(x,t) = 0$ a.e. $(x,t) \in D \times (0,\tau)$ which implies that $v(x,t) \geq 0$ a.e. $(x,t) \in D \times (0,\tau)$.
\end{proof}

Notice that Proposition \ref{prop.existence} only gives us Hölder continuity in the interior of our domain. We will now prove that our solution is also continuous up to the boundary, that is, $C_0(\Omega)$. 

\begin{lemma} \label{lemma.boundarycontinuity}
	Let $u$ be the solution to system \eqref{1.1}, with $T$ its maximal definition time. Then $u(\cdot,t) \in C_0(\Omega)$ for every $t\in(0,T)$.
\end{lemma}

\begin{proof}
	Fix some $T^\prime \in [0,T)$, and define $\delta>0$ such that 
	\begin{equation} \label{def_udelta}
	\delta = 1 - \|u\|_{L^\infty(\Omega\times[0,T^\prime])}.
	\end{equation}
	Then 
	\begin{equation} \label{eq_fproperty}
		f(u(x,t)) = \lambda (1-u(x,t))^{-p} \leq \lambda \delta^{-p}
	\end{equation}
	for every $(x,t) \in \Omega \times(0,T^\prime]$.
	
	Now take some small $\eta^\prime \in (0,1)$ and define 
	\begin{equation*}
		\Omega_{\eta^\prime} = \{x\in\Omega: \text{dist}(x,\mathbb{R}^n\setminus\Omega) <  \eta^\prime\}. 
	\end{equation*}
	Next, consider a function $d(x)$ with $x\in\mathbb{R}^n$ such that $d\in C^{2+\gamma}(\overline\Omega)$, $d(x)=0$ for every $x\in\mathbb{R}^n\setminus\Omega$, $d(x) = \text{dist}(x, \mathbb{R}^n \setminus \Omega)$ for every $x\in \Omega_{\eta^\prime}$ and $d(x) \geq \text{dist}(x, \mathbb{R}^n \setminus \Omega)$ for every $x\in \Omega \setminus \Omega_{\eta^\prime}$. Then fix any $\beta \in (1/2,1)$. 
	
	We will estimate the quantity $|(-\Delta)^s d^\beta(x)|$ for $x \in \Omega_\eta$ with $\eta < \eta^\prime/2$. First, divide the fractional laplacian as usual in a near-field integral and a far-field integral:
	\begin{equation} \label{eq_fracsplit}
		(-\Delta)^s d^\beta(x) = C_{n,s} P.V. \int_{|x-y|\leq d(x)/2} \frac{d^\beta(x)-d^\beta(y)}{|x-y|^{n+2s}} dy + C_{n,s} \int_{|x-y| > d(x)/2} \frac{d^\beta(x)-d^\beta(y)}{|x-y|^{n+2s}} dy = \mathcal{I}_n + \mathcal{I}_f.
	\end{equation}
	
	Now, let us focus on $\mathcal{I}_n$. Since $x \in \Omega_\eta$, this implies that $y \in \Omega_{\eta^\prime}$ and
	\begin{equation*}
		\frac{d(x)}{2} \leq d(y) \leq \frac{3 d(x)}{2}.
	\end{equation*}
	Furthermore, we know that $|\nabla d(x)| = 1$ and $\Delta d(x)$ is bounded for every $x\in\Omega_{\eta^\prime}$, because $d$ coincides with the distance function there. Therefore,
	\begin{equation*}
		|\Delta d^\beta(x)| \leq \beta d^{\beta-1} |\Delta d(x)| + \beta |\beta-1| d^{\beta-2}(x),
	\end{equation*}
	which, making $\eta$ possibly smaller, implies
	\begin{equation*}
		|\Delta d^\beta(x)| \leq C d^{\beta-2}(x),
	\end{equation*}
	for every $x \in \Omega_\eta$. This in turn implies that, for every $x\in \Omega_\eta$ and $z \in B_{d(x)/2}$,
	\begin{equation*}
		|2 d^\beta(x) - d^\beta (x+z) - d^\beta (x-z)| \leq C d^{\beta-2}(x) |z|^2.
	\end{equation*}
	Finally, using the symmetrized representation of the fractional laplacian, for any $x\in\Omega_\eta$:
	\begin{align} \label{eq_Inestimate}
		\nonumber|\mathcal{I}_n(x)| \leq& \; \frac{C_{n,s}}{2} \int_{|z|\leq d(x)/2} \frac{|2 d^\beta(x)-d^\beta(x+z) - d^\beta(x-z)|}{|z|^{n+2s}} dz \\
		\leq& \; C d^{\beta-2}(x) \int_{|z|\leq d(x)/2} |z|^{-n-2s+2} dz \\
		\nonumber =& \; C d^{\beta - 2s} (x),
	\end{align}
	where we have used the fact that $s<1$ for the last integral.
	
	Next, we estimate $\mathcal{I}_f$. First, we notice that, for every $x\in\Omega_\eta$,
	\begin{align*}
		|\mathcal{I}_f(x)| \leq& \; C_{n,s} \int_{|x-y| > d(x)/2} \frac{d^\beta(x)}{|x-y|^{n+2s}} dy + C_{n,s} \int_{|x-y| > d(x)/2} \frac{d^\beta(y)}{|x-y|^{n+2s}} dy \\
		\leq& \; C_{n,s} d^\beta(x) \int_{|z| > d(x)/2} |z|^{-n-2s} dz + C_{n,s} \int_{|x-y| > d(x)/2} \frac{d^\beta(y)}{|x-y|^{n+2s}} dy \\
		\leq& \; C_{n,s} d^{\beta-2s}(x) + C_{n,s} \int_{|x-y| > d(x)/2} \frac{d^\beta(y)}{|x-y|^{n+2s}} dy.
	\end{align*}
	For the second integral, split it again in two parts:
	\begin{equation} \label{eq_Ifmiddleground}
		\int_{|x-y| > d(x)/2} \frac{d^\beta(y)}{|x-y|^{n+2s}} dy = \int_{\Omega \setminus \Omega_{\eta^\prime}} \frac{d^\beta(y)}{|x-y|^{n+2s}} dy + \int_{\substack{{y \in \Omega_{\eta^\prime}} \\{|x-y| > d(x)/2}}} \frac{d^\beta(y)}{|x-y|^{n+2s}} dy.
	\end{equation}
	First, we know that if $x \in \Omega_\eta$ and $y\in \Omega \setminus \Omega_{\eta^\prime}$, then
	\begin{equation*}
		|x-y| \geq \eta,
	\end{equation*}
	so the first integral can be estimated as
	\begin{equation} \label{eq_Ifpart1}
		\int_{\Omega \setminus \Omega_{\eta^\prime}} \frac{d^\beta(y)}{|x-y|^{n+2s}} dy \leq \|d\|_{L^\infty(\Omega)}^\beta \eta^{-n-2s} \int_{\Omega \setminus \Omega_{\eta^\prime}} dy,
	\end{equation}
	which is uniformly bounded in $x$ thanks to the fact that $\Omega$ is a bounded domain. Next, notice that if $x\in\Omega_\eta$, $y\in\Omega_{\eta^\prime}$, and $|x-y| > d(x)/2$, then
	\begin{equation*}
		d(y) \leq d(x) + |x-y| \leq 3|x-y|.
	\end{equation*}
	This, together with the fact that $\Omega_{\eta^\prime}$ is bounded, implies that
	\begin{align*} 
	\int_{\substack{{y \in \Omega_{\eta^\prime}} \\{|x-y| > d(x)/2}}} \frac{d^\beta(y)}{|x-y|^{n+2s}} dy \leq& \; C \int_{\substack{{y \in \Omega_{\eta^\prime}} \\{|x-y| > d(x)/2}}} |x-y|^{-n-2s+\beta} dy \\
	\leq& \; C \int_{\substack{{|z| \leq R} \\{|z| > d(x)/2}}} |z|^{-n-2s+\beta} dz \\
	\leq& \; C \int_{d(x)/2}^R \rho^{-1-2s+\beta} d\rho.
	\end{align*}
	Thus, we conclude that
	\begin{equation} \label{eq_Ifpart2}
		\int_{\substack{{y \in \Omega_{\eta^\prime}} \\{|x-y| > d(x)/2}}} \frac{d^\beta(y)}{|x-y|^{n+2s}} dy \leq \left\{\begin{array}{ll}
		C d^{\beta-2s}(x), & \text{ if } \beta <2s, \\
		C (1+ |\log d(x)|), & \text{ if } \beta =2s, \\
		C, & \text{ if } \beta >2s.	
\end{array}		 \right.
	\end{equation}
	Finally, putting together \eqref{eq_Inestimate}, \eqref{eq_Ifmiddleground}, \eqref{eq_Ifpart1} and \eqref{eq_Ifpart2} into \eqref{eq_fracsplit}, we conclude
	\begin{equation} \label{eq_fracfinalestimate}
		|(-\Delta)^s d^\beta(x)| \leq \left\{\begin{array}{ll}
		C d^{\beta-2s}(x), & \text{ if } \beta <2s, \\
		C (1+ |\log d(x)|), & \text{ if } \beta =2s, \\
		C, & \text{ if } \beta >2s,
\end{array}		 \right.
	\end{equation}
	for every $x \in \Omega_\eta$.
	
	Then we recall that, for every $x\in\Omega_\eta$, $|\nabla d(x)|=1$ and $\Delta d(x)$ is bounded, so
	\begin{equation*}
		\Delta d^\beta(x) = \beta(\beta-1) d^{\beta-2}(x) + \beta d^{\beta-1}(x) \Delta d(x) \leq -\frac{\beta(1-\beta)}{2} d^{\beta-2}(x),
	\end{equation*}
	where maybe we need to take $\eta$ even smaller. Together with \eqref{eq_fracfinalestimate}, we conclude that
	\begin{equation}
		\mathcal{L} d^\beta (x) = \Delta d^\beta (x) - (-\Delta)^s d^\beta (x) \leq -\frac{\beta(1-\beta)}{4} d^{\beta-2}(x) \leq -\frac{\beta(1-\beta)}{4} \eta^{\beta-2} = -m_\eta, 
	\end{equation}
	for every $x\in\Omega_\eta$.
	
	Now we choose some $\psi \in C_c^\infty(\mathbb{R}^n)$ such that $0\leq \psi(x) \leq 1$ for every $x\in\mathbb{R}^n$ and $\psi \equiv 1$ in a neighbourhood of $\overline\Omega$. We first notice that, for every $x\in\Omega$, $\Delta \psi(x) = 0$ and
	\begin{equation*}
		(-\Delta)^s\psi(x) = C_{n,s} \int_{\mathbb{R}^n} \frac{1-\psi(y)}{|x-y|^{n+2s}} dy \geq 0,
	\end{equation*}
	so we conclude that $\mathcal{L} \psi (x) \leq 0$ for every $x\in\Omega$. Next, define
	\begin{equation*}
		U = 1 + \|u_0\|_\infty, \;\;\;\; C_k \geq \max \left\{(1-\delta) \eta^{-\beta}, \frac{\lambda \delta^{-p} + kU}{m_\eta} \right\},
	\end{equation*}
	for every $k\in\mathbb{N}$, and define
	\begin{equation*}
		w_k(x,t) = C_k d^\beta(x) + U e^{-kt} \psi(x),
	\end{equation*}
	for every $(x,t) \in \mathbb{R}^n \times [0,T^\prime]$ and $k\in\mathbb{N}$. 
	
	Let us prove that $z_k(x,t) = w_k(x,t)-u(x,t)$ satisfies the maximum principle of Lemma \ref{lema_maximumweak}. First, the regularity is satisfied thanks to the regularity of the solution $u$ and the fact that $\beta>1/2$, which implies that $d^\beta \in H^1(\mathbb{R}^n)$, so $w_k \in C((0,T), H^1(\mathbb{R}^n))$ for every $k\in\mathbb{N}$. Next, we can see that, for every $(x,t)\in\Omega_\eta \times (0,T^\prime)$,
	\begin{equation*}
		z_k(x,0) = C_k d^\beta(x) + U - u_0(x) \geq 1 + \|u_0\|_\infty - u_0(x) \geq 0.
	\end{equation*}
	Then, for every $(x,t) \in (\Omega \setminus \Omega_\eta) \times [0,T^\prime]$,
	\begin{equation*}
		z_k(x,t) = C_k d^\beta(x) + U e^{-kt} - u(x,t) \geq C_k \eta^\beta - u(x,t) \geq 1-\delta -u(x,t) \geq 0,
	\end{equation*}
	where we have used the definition \eqref{def_udelta}. Next, for every $(x,t) \in (\mathbb{R}^n \setminus \Omega) \times [0,T^\prime]$,
	\begin{equation*}
		z_k(x,t) = U e^{-kt} \psi(x) \geq 0.
	\end{equation*}
	Finally, let us see that $z_k$ satisfies the differential inequality for every $(x,t) \in \Omega_\eta \times (0,T^\prime)$:
	\begin{align*}
		(z_k)_t(x,t) - \mathcal{L} z_k(x,t) \geq& \; -k U e^{-kt} - u_t(x,t) + C_k m_\eta + \mathcal{L} u(x,t) \\
		\geq& \; -k U + (\lambda \delta^{-p} + kU) - \lambda(1-u(x,t))^{-p} \\
		=& \; \lambda \delta^{-p} - \lambda(1-u(x,t))^{-p} \geq 0,
	\end{align*}
	where we have used \eqref{eq_fproperty} for the last inequality. Notice that these inequalities are satisfied pointwise, so they are also satisfied in the weak sense. Then, we are under the conditions of Lemma \ref{lema_maximumweak} and can conclude that
	\begin{equation*}
		w_k(x,t) \geq u(x,t),
	\end{equation*}
	a.e. $(x,t) \in \Omega_\eta \times (0,T^\prime)$ for every $k\in\mathbb{N}$. By the continuity of both $w_k$ and $u$, it is satisfied for every $(x,t) \in \Omega_\eta \times (0,T^\prime)$. Now fix any $t_0 \in (0,T^\prime)$. For every $x\in\Omega_\eta$ and $k\in\mathbb{N}$, we have
	\begin{equation*}
		C_k d^\beta(x) + U e^{-k t_0} \geq u(x,t_0) \geq 0.
	\end{equation*}
	In $\Omega_\eta$, $d(x) = \text{dist} (x,\mathbb{R}^n \setminus \Omega)$, so taking the limit $x \to x_0 \in \partial\Omega$:
	\begin{equation*}
		U e^{-k t_0} \geq \limsup_{x\to x_0} u(x,t_0) \geq 0.
	\end{equation*}
	Finally, since this is valid for every $k\in\mathbb{N}$, by taking $k\to\infty$ we conclude that
	\begin{equation*}
		\limsup_{x\to x_0} u(x,t_0) = 0,
	\end{equation*}
	uniformly in $x_0 \in \partial\Omega$. Then $u(\cdot,t_0)$ extends continuously to the boundary and we can conclude that
	\begin{equation*}
		u(\cdot,t_0) \in C_0(\Omega).
	\end{equation*}
	Since $t_0 \in (0,T^\prime)$ was arbitrary, and $T^\prime \in (0,T)$ was as well, we finally conclude
	\begin{equation*}
		u(\cdot,t) \in C_0(\Omega),
	\end{equation*}
	for every $t\in(0,T)$.
\end{proof}

\begin{remark}
	Notice that the barrier we have constructed for this result is, barring the exponential term, $C_k d^\beta (x)$, and the constant $C_k$ depends drastically on $\delta$, which in turn depends on the choice of $T^\prime$. Therefore, this constant could explode to $+\infty$ as $t\to T$, which would allow the solution to present quenching at a boundary point. We will prove this is not possible in Theorem \ref{teo.quenchingset}.
\end{remark}
Now that we have proved that solutions of \eqref{1.1} are continuous up to the boundary, we are able to define pointwise supersolutions and subsolutions and prove a pointwise comparison result.
	
	\begin{definition} \label{def_supersolution_particularized}
    	We say $\overline{u}\in C^{1,\alpha/2}_{loc}((0,T), C^\alpha(\Omega)) \cap  C([0,T),L^1_{2s}(\mathbb{R}^n))$ is a (pointwise) supersolution of system \eqref{1.1} in $[t_0,t_1] \subset [0,T)$ if $\overline{u}(\cdot,t) \in C(\overline\Omega) \cap C^{2,\alpha}_{loc}(\Omega)$ for every $t\in [t_0,t_1]$, $\overline{u}(x,t) \leq 1 - \delta$ for every $(x,t) \in \Omega \times [t_0,t_1]$ and some $\delta \in (0,1)$, and the following inequalities are satisfied pointwise:
    	 \begin{equation}
    \left\{
    \begin{array}{ll}
        \displaystyle \overline{u}_t(x,t)\geq \mathcal{L}\overline{u}(x,t)+ \lambda (1-\overline{u}(x,t))^{-p}, & (x,t)\in\Omega\times [t_0,t_1], \\
        \overline{u}(x,t_0)\geq u(x,t_0), & x\in\Omega,\\
        \overline{u}(x,t) \geq 0, & (x,t)\in (\mathbb{R}^n \setminus \Omega) \times [t_0,t_1].
    \end{array}
    \right.
    \end{equation}
    We say $\underline{u}\in C^{1,\alpha/2}_{loc}((0,T), C^\alpha(\Omega)) \cap  C([0,T),L^1_{2s}(\mathbb{R}^n))$ is a (pointwise) subsolution of system \eqref{1.1} in $[t_0,t_1] \subset [0,T)$ if $\underline{u}(\cdot,t) \in C(\overline\Omega) \cap C^{2,\alpha}_{loc}(\Omega)$ for every $t\in [t_0,t_1]$, $\underline{u}(x,t) \leq 1-\delta$ for every $(x,t) \in \Omega \times [t_0,t_1]$ and some $\delta \in (0,1)$, and the reverse inequalities are satisfied pointwise.
    \end{definition}
	
	\begin{lemma} \label{lemma_comparison_particularized}
    	Let $\overline{u}$ be a (pointwise) supersolution of \eqref{1.1} in $[t_0,t_1]$ and $\underline{u}$ a (pointwise) subsolution of \eqref{1.1} in $[t_0,t_1]$. Then
    	\begin{equation*}
    		\overline{u}(x,t) \geq \underline{u}(x,t)
    	\end{equation*}
    	for every $(x,t)\in \Omega \times [t_0,t_1]$.
    \end{lemma}	
    
    \begin{proof}
    	Set $w(x,t) = \overline{u}(x,t) - \underline{u}(x,t)$, and notice that, thanks to the Mean Value Theorem,
    	\begin{equation*}
    		\lambda (1-\overline{u}(x,t))^{-p} - \lambda (1-\underline{u}(x,t))^{-p} = \lambda p (1-\xi(x,t))^{-p-1} w(x,t),
    	\end{equation*}
    	where $\xi(x,t)$ is a function that takes values between $\underline{u}(x,t)$ and $\overline{u}(x,t)$ for every $(x,t) \in \Omega \times[t_0,t_1]$. Thanks to Definition \ref{def_supersolution_particularized}, we know that there exists some $\delta \in (0,1)$ such that $\xi(x,t) \leq 1-\delta$ for every $(x,t) \in \Omega \times [t_0,t_1]$. Therefore:
    	\begin{equation} \label{eq_derivativebound}
    		\lambda p (1-\xi(x,t))^{-p-1} \leq \lambda p \delta^{-p-1} = C_\delta.
    	\end{equation}
    	
    	Then, define $z(x,t) = e^{-C_\delta t} w(x,t)$ for every $(x,t) \in \Omega \times [t_0,t_1]$ and suppose that there exists a point $(x_0,\tau_0) \in \overline\Omega \times [t_0,t_1]$ in which $z$ attains a negative minimum for the first time. Then, since $z(x,t) \geq 0$ for every $(x,t) \in (\mathbb{R}^n \setminus \Omega) \times [t_0,t_1]$, we know that $x_0 \in \Omega$. This leads to
    	\begin{equation*}
    		z_t(x_0,\tau_0) \leq 0, \;\;\;\; -\Delta z(x_0,\tau_0) \leq 0, \;\;\;\; (-\Delta)^s z(x_0,\tau_0) < 0,
    	\end{equation*}
    	where we have used the fact that $z(x_0,\tau_0)$ is strictly negative and $z(x,t) \geq 0$ for every $(x,t) \in (\mathbb{R}^n \setminus \Omega) \times [t_0,t_1]$ for the last inequality. Then:
    	\begin{align*}
    		z_t(x_0,\tau_0) =& \; e^{-C_{\delta} \tau_0} w_t(x_0,\tau_0) - C_\delta e^{-C_\delta \tau_0} w(x_0,\tau_0) \\
    		\geq& \; e^{-C_\delta \tau_0} \left( \Delta w(x_0,\tau_0) -(-\Delta)^s w(x_0,\tau_0) + \lambda p (1-\xi(x_0,\tau_0))^{-p-1} w(x_0,\tau_0) - C_\delta w(x_0,\tau_0) \right) \\
    		>& \; e^{-C_\delta \tau_0} \left( \lambda p (1-\xi(x_0,\tau_0))^{-p-1} w(x_0,\tau_0) - C_\delta w(x_0,\tau_0) \right) \geq 0,
    	\end{align*}
    	where we have used the fact that $w(x_0,\tau_0) <0$ and \eqref{eq_derivativebound} for the last inequality. Then, we conclude that $z_t(x_0,\tau_0) >0$, which is a contradiction. Therefore, $z(x,t) \geq 0$ for every $(x,t) \in \Omega \times [t_0,t_1]$, which implies that $w(x,t) \geq 0$ for every $(x,t) \in \Omega \times [t_0,t_1]$, and the result is proved.
    \end{proof}

    \begin{lemma} \label{lema.velocidadpositiva}
    Let $u$ be the solution of \eqref{1.1} with initial datum $u_0$ such that
    \begin{equation} \label{cond_u0subsolution}
        \mathcal{L} u_0(x) + \lambda (1-u_0(x))^{-p} \geq 0
    \end{equation}
    in the weak sense. Then $u(x,t) \geq u_0 (x)$ a.e. $x \in \Omega$ for every $t\in (0,T)$ and $u_t(x,t) \geq 0$ for every $(x,t) \in \Omega\times (0,T)$.
\end{lemma}

\begin{proof}
    For the first part, define $w(x,t) = u(x,t) - u_0(x)$ for every $(x,t) \in \Omega \times [0,T)$. It is easy to prove that this function is under the hypotheses of Lemma \ref{lema_maximumweak}, so we know that $u(x,t) \geq u_0(x)$ a.e. $x\in\Omega$ for every $t\in[0,T)$, where we have taken into account the continuity in time of $u$ to obtain an inequality for every time.

    To prove the second part, consider the system satisfied by $v_h = \frac{u(x,t+h)-u(x,t)}{h}$:
    \begin{equation*}
    \left\{
        \begin{array}{ll}
             \displaystyle (v_h)_t (x,t) = \mathcal{L} v_h + a_h(x,t) v_h (x,t), &(x,t)\in\Omega\times(0,T-h), \\
             v_h(x,t) = 0, &(x,t)\in (\mathbb{R}^n \setminus \Omega) \times[0,T-h), \\
             v_h(x,0) \geq 0, &x\in\Omega,
        \end{array}
        \right.
    \end{equation*}
    where
    \begin{equation} \label{eq_ahdef2}
    	a_h(x,t) = \int_0^1 f^\prime ((1-\theta)u(x,t)+\theta u(x,t+h)) d\theta,
    \end{equation}
	where in this case
	\begin{equation*}
	f^\prime (s) = \lambda p (1-s)^{-p-1}.
	\end{equation*}
	Take some $\tau \in (0,T-h)$ and
	\begin{equation*}
		1-\delta = \sup_{t\in[0,\tau+h]} \|u(\cdot,t)\|_{L^\infty(\Omega)} < 1.
	\end{equation*} 
	Then we see that $0\leq a_h(x,t) \leq M_\delta$ for every $(x,t) \in \Omega \times [0,\tau]$ and some $M_\delta >0$. Then $v_h$ is under the conditions of Lemma \ref{lema_maximumweak} and $v_h(x,t) \geq 0$ a.e. $x\in\Omega$ for every $t\in(0,\tau]$.
	
	Let us now prove that $v_h$ converges strongly to $u_t$ in $C([T_1,T_2], L^2(\Omega))$ for every $[T_1,T_2] \subset (0,T-h)$. Indeed, for any $t \in [T_1,T_2]$, we know that
    \begin{align*}
    	\|v_h(\cdot,t) - u_t(\cdot,t) \|_{L^2(\Omega)} \leq& \; \frac{1}{h} \int_{t}^{t+h} \|u_t(\cdot,s) - u_t(\cdot,t)\|_{L^2(\Omega)} ds \\
    	\leq& \; \sup_{\substack{t_1,t_2 \in [T_1, T_2+h]\\ |t_1-t_2| \leq h}} \|u_t(\cdot,t_1) - u_t(\cdot,t_2)\|_{L^2(\Omega)} .
    \end{align*}
    Since $u_t \in C((0,T), L^2(\Omega))$, it is uniformly continuous in the compact interval $[T_1,T_2+h]$. Then, the right hand side tends to $0$ as $h\to 0$. With this we conclude
    \begin{equation} \label{eq_strongvhconvergence}
    	v_h \rightarrow u_t \;\;\; \text{ in } \; C([T_1,T_2], L^2(\Omega)).
    \end{equation}
  	Then, passing $h\to 0$ and using the strong convergence, we obtain that $u_t(x,t) \geq 0$ a.e. $x\in\Omega$ for every $t\in[\tau_0,\tau]$ with $\tau_0 \in (0,\tau)$. Since the argument is valid for every $[\tau_0,\tau]\in(0,T-h)$ and every $h>0$, and using the Hölder continuity of $u_t$, we get that $u_t(x,t)\geq 0$ for every $(x,t) \in \Omega \times(0,T)$ and the result is proved.
\end{proof}
    
    \begin{corollary} \label{cor_positivity_particularized}
    	The solution $u$ of \eqref{1.1} satisfies $u(x,t) > 0$ for every $(x,t) \in \Omega \times (0,T)$.
    \end{corollary}
	
	\begin{proof}
		Thanks to Lemma \ref{lema.velocidadpositiva}, we know that $u(x,t) \geq u_0(x) \geq 0$ a.e. $x\in\Omega$ for every $t\in(0,T)$, and thanks to the continuity of the solution $u$, we get that $u(x,t) \geq 0$ for every $(x,t) \in \Omega \times (0,T)$. Then we suppose that there exists some time $t_0>0$ in which $u$ touches $0$ at some $x_0\in\Omega$. Therefore:
    \begin{equation*}
        u(x_0,t_0) = 0, \;\;\; u_t(x_0,t_0) \leq0, \;\;\; \Delta u(x_0,t_0) \geq 0, \;\;\; -(-\Delta)^s u(x_0,t_0) \geq 0.
    \end{equation*}
    Plugging these inequalities into the equation for \eqref{1.1}, we get that
    \begin{align*}
        0 \geq u_t(x_0,t_0) = -(-\Delta)^s u(x_0,t_0) +\Delta u(x_0,t_0)+\lambda (1-u)^{-p}(x_0,t_0) > 0,
    \end{align*}
    which leads to a contradiction. Therefore, $u(x,t) > 0$ for every $(x,t) \in \Omega \times(0,T)$.
	\end{proof}
	
	\subsection{Existence and regularity of odd eigenfunctions}

Next, we are interested in replicating the results of \cite{BDVV2023} to find an odd eigenfunction of $-\mathcal{L}$ in a symmetric domain, which will be a crucial ingredient in applications of the method of moving planes stated in Theorem \ref{teo.movingplanes}.

First, let us define the symmetry planes with which we will work throughout the rest of the paper. Let $\Psi$ be a hyperplane in $\mathbb{R}^n$. Without loss of generality, take some $l\in\mathbb{R}$ and assume that
    \begin{equation*}
        \Psi = \{x=(x_1,x^\prime) \in \mathbb{R}^n : x_1 = l\}.
    \end{equation*}
    Let
    \begin{equation*}
        \hat{x} = (2l-x_1,x^\prime), \;\;\; H = \{x=(x_1,x^\prime)\in\mathbb{R}^n : x_1 < l\}, \;\;\; \widehat{H} = \{x\in\mathbb{R}^n : \hat{x}\in H\},
    \end{equation*}    
    which are the reflection of $x$ with respect to $\Psi$, the left half-space and the right half-space respectively.

\begin{lemma} \label{lemma.oddeigenf}
	Consider an open bounded domain $D \subset \mathbb{R}^n$ with $C^{1,1}$ boundary and symmetric with respect to $\Psi$. Define $D^- = D \cap H$, $D^+ = D \cap \widehat{H}$, and consider the space
	\begin{equation*}
		X_{odd} = \{ u_{|D} \in H_0^1 (D) : u \equiv 0 \; \text{ in } \; \mathbb{R}^n \setminus D, \;\;\; u(x) = -u(\hat{x}) \; \text{ in } \; D^- \}.
	\end{equation*}
	
	We define the principal odd Dirichlet eigenvalue of $-\mathcal{L}$ in $D$ as 
	\begin{equation} \label{eq.oddeigenv}
		\mu_{odd} = \inf \left\{ \frac{\mathcal{E}(u,u)}{\|u\|_{L^2(D)}^2} : u \in X_{odd} \setminus \{0\} \right\} \in [0,\infty).
	\end{equation}
	
	Then there exists some $\varphi_{odd} \in X_{odd} \setminus \{0\}$ such that 
	\[
		\|\varphi_{odd}\|_{L^2(D)} = 1, \, 
		\mu_{odd} = \mathcal{E}(\varphi_{odd},\,
		\varphi_{odd})>0,\text{ and }
		\varphi_{odd} < 0
		\text{ in }D^-.
	\] 
	Furthermore, $\varphi_{odd} \in C_b(\mathbb{R}^n) \cap C^{2,\alpha}_{loc}(D) \cap C^{1,\beta}(\overline D)$ and it is a classical solution of the problem
	
	\begin{equation} \label{eq.oddeigenfunction}
    \left\{
    \begin{array}{ll}
        -\mathcal{L}\varphi_{odd}(x) = \mu_{odd} \varphi_{odd}(x), & x\in D,\\
        \varphi_{odd}(x)=0, & x\in \mathbb{R}^n \setminus D,\\
        \varphi_{odd}(x) = -\varphi_{odd}(\hat{x}), & x \in D^-.
    \end{array}
    \right.
    \end{equation}
\end{lemma}

\begin{proof}
	Notice that $X_{odd}$ is a closed subspace of $H_0^1(D)$, so we can follow the proof of \cite[Proposition 9]{SV2013} to get that the infimum in \eqref{eq.oddeigenv} is achieved and therefore $\mu_{odd}$ is an eigenvalue with an associated eigenfunction $e_{odd} \in X_{odd}$ with $\|e_{odd}\|_{L^2(D)}=1$.
	
	Suppose then that $\mu_{odd} = 0$. Then,
	\begin{equation*}
		\mathcal{E}(e_{odd}, e_{odd}) = \frac{C_{n,s}}{2} \iint_{\mathbb{R}^n \times \mathbb{R}^n} \frac{|e_{odd}(x)-e_{odd}(y)|^2}{|x-y|^{n+2s}} dy dx + \int_D |\nabla e_{odd}(x)|^2 dx = 0,
	\end{equation*}	
	which means that $e_{odd}(x)$ is constant a.e. $x\in\mathbb{R}^n$. But since $e_{odd} \in X_{odd}$, we know that $e_{odd}(x) = 0$ a.e. $x\in \mathbb{R}^n \setminus D$, which in turn gives us that $e_{odd} \equiv 0$, which is a contradiction with the fact that $\|e_{odd}\|_{L^2(D)}=1$. Then we conclude that $\mu_{odd}>0$.
	
	 Let us now prove that we can choose this eigenfunction so that $e_{odd}\leq 0$ in $D^-$. We claim that, if $e \in X_{odd}$ is an eigenfunction related to $\mu_{odd}$ with $\|e\|_{L^2(\mathbb{R}^n)} = 1$, then both $e$ and $\tilde{e}=sign(x_1-l)|e|$ attain the minimum in \eqref{eq.oddeigenv}, and either $e\geq 0$ or $e\leq 0$ in $D^-$.
	
	Indeed, we know that
	\begin{equation*}
		\mathcal{E}(e,e) = \int_D |\nabla e|^2 dx + \frac{C_{n,s}}{2} \iint_{\mathbb{R}^n \times \mathbb{R}^n} \frac{(e(x)-e(y))^2}{|x-y|^{n+2s}} dx dy = \mu_{odd} = \mathcal{E}(e_{odd},e_{odd}),
	\end{equation*}
	and, for any $v\in X_{odd}$,
	\begin{align*}
		\mathcal{E}(v,v) =& \; \int_D |\nabla v|^2 dx + \frac{C_{n,s}}{2}\iint_{\mathbb{R}^n \times \mathbb{R}^n} \frac{(v(x)-v(y))^2}{|x-y|^{n+2s}} dx dy \\
		=& \;  \int_D |\nabla v|^2 dx + \frac{C_{n,s}}{2} \left( 2 \int_{D} \int_{\mathbb{R}^n\setminus D} \frac{v(y)^2}{|x-y|^{n+2s}} dx dy + \int_D \int_D \frac{(v(x)-v(y))^2}{|x-y|^{n+2s}} dx dy \right) \\
		=& \;  \int_D |\nabla v|^2 dx + C_{n,s} \int_{D} \int_{\mathbb{R}^n\setminus D} \frac{v(y)^2}{|x-y|^{n+2s}} dx dy + \frac{C_{n,s}}{2} \mathcal{I}(v).
	\end{align*}
	For the integral $\mathcal{I}(v)$ we know that, thanks to the symmetry,
	\begin{align*}
	\mathcal{I}(v) =& \; \int_{D^+} \int_{D^+} \frac{(v(x)-v(y))^2}{|x-y|^{n+2s}} dx dy + \int_{D^-} \int_{D^-} \frac{(v(x)-v(y))^2}{|x-y|^{n+2s}} dx dy \\
	&+ \int_{D^+} \int_{D^-} \frac{(v(x)-v(y))^2}{|x-y|^{n+2s}} dx dy + \int_{D^-} \int_{D^+} \frac{(v(x)-v(y))^2}{|x-y|^{n+2s}} dx dy \\
	=& \; 2\int_{D^-} \int_{D^-} \frac{(v(x)-v(y))^2}{|x-y|^{n+2s}} dx dy + 2\int_{D^-} \int_{D^+} \frac{(v(x)-v(y))^2}{|x-y|^{n+2s}} dx dy \\
	=& \; 2\int_{D^-} \int_{D^-} \left( \frac{(v(x)-v(y))^2}{|x-y|^{n+2s}} + \frac{(v(x)+v(y))^2}{|x-\hat{y}|^{n+2s}} \right) dx dy \\
	=& \; 2\int_{D^-} \int_{D^-} (v(y)^2+v(x)^2) \left( \frac{1}{|x-y|^{n+2s}} + \frac{1}{|x-\hat{y}|^{n+2s}} \right) dxdy\\
	&+  4 \int_{D^-} \int_{D^-} v(x)v(y)\left( \frac{-1}{|x-y|^{n+2s}} + \frac{1}{|x-\hat{y}|^{n+2s}} \right) dx dy.
	\end{align*}
	Then, since both $e$ and $\tilde{e}$ are functions in $X_{odd}$, $e(x)^2 = \tilde{e}(x)^2$ and $|\nabla e(x)|^2=|\nabla \tilde{e}(x)|^2$ a.e. $x\in D$, and  $|x-y|<|x-\hat{y}|$ for every $x,y \in D^-$,
	\begin{align*}
		\mathcal{E}(\tilde{e},\tilde{e}) &- \mathcal{E}(e,e) \\
		 =& \; 2 C_{n,s} \int_{D^-} \int_{D^-} (\tilde{e}(x)\tilde{e}(y) - e(x)e(y)) \left( \frac{-1}{|x-y|^{n+2s}} + \frac{1}{|x-\hat{y}|^{n+2s}} \right) dx dy \\
		=& \; 2 C_{n,s} \int_{D^-} \int_{D^-} (|e(x)||e(y)| - e(x)e(y)) \left( \frac{-1}{|x-y|^{n+2s}} + \frac{1}{|x-\hat{y}|^{n+2s}} \right) dx dy \\
		\leq& \; 0,
	\end{align*}
	and the inequality is strict if both $\{e>0\} \cap D^-$ and $\{e<0\} \cap D^-$ have positive measure. Finally, the fact that $\tilde{e}\in X_{odd}$, $\|\tilde{e}\|_{L^2(D)} = \|e\|_{L^2(D)} =1$ and $e_{odd}$ is a minimizer of $\mathcal{E}(\cdot,\cdot)$, we conclude that $\mathcal{E}(\tilde{e},\tilde{e}) = \mathcal{E}(e,e) = \mathcal{E}(e_{odd},e_{odd})$ and either $e\geq 0$ or $e\leq 0$ in $D^-$. This proves the claim and lets us conclude that we can choose the odd eigenfunction $e_{odd}$ such that $e_{odd} \leq 0$ in $D^-$.
	
	For the regularity result and the proof that the eigenfunction is a classical solution of \eqref{eq.oddeigenfunction}, we refer to \cite[Theorems 2.7 and 3.5]{BDVV2023} and \cite[Theorem 1.5]{SVWZ2025}.
	
	Finally, we state that this choice of $e_{odd}$ satisfies $e_{odd}<0$ in $D^-$. Indeed, suppose that there exists a point $x_0\in D^-$ such that $e_{odd}(x_0)=0$. Then it is clear that $\Delta e_{odd}(x_0) \leq 0$ and, following a symmetry argument similar to the one done before, $(-\Delta)^s e_{odd} (x_0) \geq 0$, and the inequality is strict unless $e_{odd}\equiv 0$. However, that cannot be true because $\|e_{odd}\|_{L^2(D)}=1$, so the inequality is actually strict. Then plugging this into \eqref{eq.oddeigenfunction}, we get
	\begin{equation*}
		0 < - \mathcal{L} e_{odd} (x_0) = \mu_{odd} e_{odd}(x_0) = 0,
	\end{equation*}
	which is a contradiction.
\end{proof}
    
    \section{Parabolic moving planes} \label{sec.movingplanes}
   Our first main objective is to extend the method of moving planes presented in \cite{CLL2016} for a parabolic system involving our mixed operator $\mathcal{L}$. Traditionally, to apply the maximum principle or comparison results to a system with a nonlocal operator on a domain $D\subset \mathbb{R}^n$, we are forced to impose a positivity condition on the whole complement $\mathbb{R}^n\setminus D$. Notice that we are indeed using such condition both in Lemma \ref{lema_maximumweak} and Lemma \ref{lemma_comparison_particularized}. The method of moving planes takes advantage of the symmetry of the kernel $|x-y|^{-n-2s}$ to relax this hypothesis.

\begin{proof}[Proof of Theorem \ref{teo.movingplanes}]
	Assume, without loss of generality, that $\tau=0$. Take some $T^\prime \in (0,T)$ and $M>0$ such that $\|C\|_{L^\infty([0,T^\prime])} \leq M$. Then define, for every $(x,t)\in \mathbb{R}^n \times [0,T^\prime]$,
	\begin{equation} \label{integrating.factor}
	v(x,t) = e^{-Mt} u(x,t).
	\end{equation}
    Notice that the regularity of $v$ is the same as that of $u$. Furthermore, it also satisfies that $v(x,0)\geq 0$ for every $x\in D^-$, $v(x,t) \geq 0$, for every $(x,t) \in (H\setminus D^-) \times[0,T^\prime]$, and $v(x,t)=-v(\hat{x},t)$ for every $(x,t) \in H\times[0,T^\prime]$. Finally, using \eqref{cond.movingplanes}, it is easy to see that:
    \begin{align} \label{eq.integratingfactor}
    	v_t(x,t) &= e^{-Mt} u_t(x,t) - M e^{-Mt} u(x,t) \\
    	&\geq - (-\Delta)^s v(x,t) + \Delta v(x,t) + (c(x,t) - M) v(x,t), \nonumber
    \end{align}
    for every $(x,t)\in D^- \times(0,T^\prime]$. Now suppose that there exists some $(x_0,t_0)\in D^- \times(0,T^\prime]$ such that
    \begin{equation*}
    	v(x_0,t_0) = \min_{\overline{D^-}} v(\cdot,t_0) < 0,
    \end{equation*}
    and, without loss of generality,
    \begin{equation*}
    	v_t(x_0,t_0) \leq 0, \;\;\; \Delta v(x_0,t_0) \geq 0.
    \end{equation*}
    We also claim that $(-\Delta)^s v(x_0,t_0) < 0$. Indeed, taking advantage of the symmetry:
    
    \begin{align*}
    	(-\Delta)^s & v(x_0,t_0)= C_{n,s} P.V. \int_{\mathbb{R}^n} \frac{v(x_0,t_0)-v(y,t_0)}{|x_0-y|^{n+2s}} dy \\
    	&= C_{n,s} P.V. \left\{\int_H \frac{v(x_0,t_0)-v(y,t_0)}{|x_0-y|^{n+2s}} dy + \int_{\widehat{H}} \frac{v(x_0,t_0)-v(y,t_0)}{|x_0-y|^{n+2s}} dy \right\} \\
    	&= C_{n,s} P.V. \left\{\int_H \frac{v(x_0,t_0)-v(y,t_0)}{|x_0-y|^{n+2s}} dy + \int_{H} \frac{v(x_0,t_0)-v(\hat{y},t_0)}{|x_0-\hat{y}|^{n+2s}} dy \right\} \\
    	&= C_{n,s} P.V. \left\{\int_H \frac{v(x_0,t_0)-v(y,t_0)}{|x_0-y|^{n+2s}} dy + \int_{H} \frac{v(x_0,t_0)+v(y,t_0)}{|x_0-\hat{y}|^{n+2s}} dy \right\} \\
    	&\leq C_{n,s} P.V. \int_H \frac{v(x_0,t_0)-v(y,t_0) + v(x_0,t_0) + v(y,t_0)}{|x_0-\hat{y}|^{n+2s}} dy <0.
    \end{align*}
    Plugging all these into \eqref{eq.integratingfactor}, we get:
    \begin{align*}
    	0 \geq v_t(x_0,t_0) &\geq -(-\Delta)^s v(x_0,t_0) + \Delta v(x_0,t_0) + (c(x_0,t_0) - M) v(x_0,t_0) \\
    	&> (c(x_0,t_0) - M) v(x_0,t_0) \geq 0,
    \end{align*}
    which is a contradiction. Therefore, we conclude that $v(x,t) \geq 0$ for every $(x,t)\in D^- \times[0,T^\prime]$. Following \eqref{integrating.factor}, this means that $u(x,t) \geq 0$ for every $(x,t)\in D^- \times[0,T^\prime]$. Since this is valid for every $T^\prime \in (0,T)$, it is clear that $u(x,t) \geq 0$ for every $(x,t)\in D^- \times[0,T)$.
    
    To prove the strong principle, take some $(x_1,t_1)\in D^- \times (0,T)$ such that $v(x_1,t_1)=0$. Then we can assume without loss of generality that 
    \begin{equation*}
    	v_t(x_1,t_1) \leq 0, \;\;\; \Delta v(x_1,t_1) \geq 0.
    \end{equation*}
    Using \eqref{eq.integratingfactor} again, we get
    \begin{equation*}
    	0 \leq (-\Delta)^s v(x_1,t_1) = C_{n,s} \int_{H} v(y,t_1) \left(\frac{1}{|x_1-\hat{y}|^{n+2s}} -\frac{1}{|x_1-y|^{n+2s}} \right) dy \leq 0,
    \end{equation*}
    from which we derive immediately that $v(\cdot, t_1) \equiv 0$ in $\mathbb{R}^n$. Then, following Lemma \ref{lemma.oddeigenf}, consider the odd principal eigenvalue of $-\mathcal{L}$ in $D$ and its corresponding eigenfunction, that is, $\mu_D > 0$ and $\varphi_D \in C^{1,\beta}(\overline{D}) \cap C^{2,\alpha}_{loc}(D)$ that satisfy
    \begin{equation*}
    \left\{
    \begin{array}{ll}
        -\mathcal{L}\varphi_{D}(x) = \mu_{D} \varphi_{D}(x), & x\in D,\\
        \varphi_{D}(x)=0, & x\in \mathbb{R}^n \setminus D,\\
        \varphi_{D}(x) = -\varphi_{D}(\hat{x}), & x \in D.
    \end{array}
    \right.
    \end{equation*}
    Choose without loss of generality that $\varphi_D >0$ in $D^-$ (and $\varphi_D<0$ in $D^+$ in consequence), and define the quantity
    \begin{equation*}
    	w_D (t) = \int_{D^-} v(x,t) \varphi_D (x)dx.
    \end{equation*}
    for every $t\in[0,T^\prime]$. Since $v(x,t) \geq 0$ for every $(x,t) \in D^- \times [0,T)$ and $\varphi_D(x) > 0$ for every $x\in D^-$, it is immediate that $w_D(t) \geq 0$ for every $t\in[0,T^\prime]$.
    
    Next, take $x\in H\setminus D^-$ and $z \in D^-$. This implies that
    \begin{equation*}
    	|x-z| \leq |x-\hat{z}|.
    \end{equation*}
    Then, for every $x\in H\setminus D^-$:
    \begin{align*}
    	(-\Delta)^s \varphi_D(x) =& \; C_{n,s} \int_{\mathbb{R}^n} \frac{-\varphi_D(y)}{|x-y|^{n+2s}} dy \\
    	=& \; C_{n,s} \int_{D^-} \frac{-\varphi_D(y)}{|x-y|^{n+2s}} dy + C_{n,s} \int_{D^+} \frac{-\varphi_D(y)}{|x-y|^{n+2s}} dy \\
    	=& \; C_{n,s} \int_{D^-} \frac{-\varphi_D(y)}{|x-y|^{n+2s}} dy + C_{n,s} \int_{D^-} \frac{\varphi_D(y)}{|x-\hat{y}|^{n+2s}} dy \\
    	\leq& \; C_{n,s} \int_{D^-} \frac{-\varphi_D(y)}{|x-y|^{n+2s}} dy + C_{n,s} \int_{D^-} \frac{\varphi_D(y)}{|x-y|^{n+2s}} dy = 0.
    \end{align*}
    In consequence, for every $x \in \hat{H} \setminus D^+$, using the symmetry of $\varphi_D$ and the linearity of the fractional laplacian:
    \begin{equation*}
    	(-\Delta)^s \varphi_D (x) = (-\Delta)^s (-\varphi_D(\hat{x})) = -(-\Delta)^s \varphi_D(\hat{x}) \geq 0.
    \end{equation*}
    Then, we can get that, for every $t \in [0,T^\prime]$,
    \begin{align*}
    	\int_{\mathbb{R}^n \setminus D} v(x,t) (-\Delta)^s \varphi_D(x) dx =& \; \int_{H \setminus D^-} v(x,t) (-\Delta)^s \varphi_D(x) dx + \int_{\hat{H} \setminus D^+} v(x,t) (-\Delta)^s \varphi_D(x) dx \\
    	=& \; \int_{H \setminus D^-} (v(x,t) (-\Delta)^s \varphi_D(x) + v(\hat{x},t) (-\Delta)^s \varphi_D(\hat{x})) dx,
    \end{align*}
    and notice that, for every $(x,t)\in (H\setminus D^-) \times [0,T^\prime]$, we have
    \begin{equation*}
    	v(x,t) \geq 0, \;\;\; (-\Delta)^s \varphi_D(x) \leq 0, \;\;\; v(\hat{x},t) \leq 0, \;\;\; (-\Delta)^s \varphi_D (\hat{x}) \geq 0,
    \end{equation*}
    so we conclude
    \begin{equation} \label{eq_signofintegraloutside}
    	\int_{\mathbb{R}^n \setminus D} v(x,t) (-\Delta)^s \varphi_D(x) dx \leq 0,
    \end{equation}
    for every $t\in[0,T^\prime]$. Next, notice that
    \begin{align*}
    	\int_{D^-} (-\Delta)^s v(x,t) \varphi_D(x) dx = \frac{1}{2} \int_{D} (-\Delta)^s v(x,t) \varphi_D(x) dx =& \; \frac{1}{2} \int_{\mathbb{R}^n} (-\Delta)^s v(x,t) \varphi_D(x) dx \\
    	 =& \; \frac{1}{2} \int_{\mathbb{R}^n} v(x,t) (-\Delta)^s \varphi_D(x) dx,
    \end{align*}
    where we have used the fact that $\varphi_D (x) = 0$ for every $x \in \mathbb{R}^n \setminus D$ and the symmetry of the fractional laplacian. Thanks to \eqref{eq_signofintegraloutside}, we find that
    \begin{equation} \label{eq_fraclapinequality}
    \int_{D^-} (-\Delta)^s v(x,t) \varphi_D(x) dx \leq \frac{1}{2} \int_{D} v(x,t) (-\Delta)^s \varphi_D(x) dx = \int_{D^-} v(x,t) (-\Delta)^s \varphi_D(x) dx
    \end{equation}
    for every $t\in[0,T^\prime]$. At the same time, the weak Green identity gives us:
    \begin{equation*}
    	\int_{D^-} \Delta v(x,t) \varphi_D(x) dx = \int_{D^-} v(x,t) \Delta \varphi_D(x) dx - \int_{\partial D^-} Tr(v)(x,t) \partial_\nu \varphi_D(x) dx .
    \end{equation*}
    Notice that, for every $t\in[0,T^\prime]$, if $x\in D^-$, then $v(x,t) \geq0$ and $\partial_\nu \varphi_D(x) \leq 0$. This lets us conclude that
    \begin{equation} \label{eq_locallapinequality}
    	\int_{D^-} \Delta v(x,t) \varphi_D(x) dx \geq \int_{D^-} v(x,t) \Delta \varphi_D(x) dx.
    \end{equation}
    
    We can now put together \eqref{eq_fraclapinequality} and \eqref{eq_locallapinequality} to see that, for every $t\in[0,T^\prime]$,
    \begin{equation} \label{eq_mixedlapinequality}
    	\int_{D^-} (\mathcal{L} v(x,t))\varphi_D(x) dx \geq - \int_{D^-} v(x,t) (-\Delta)^s \varphi_D(x) dx + \int_{D^-} v(x,t) \Delta \varphi_D(x) dx = \int_{D^-} v(x,t) (\mathcal{L}\varphi_D(x)) dx.
    \end{equation}
    Then, differentiating $w(t)$ in time for any $t\in[0,T^\prime]$ and using \eqref{eq.integratingfactor} and \eqref{eq_mixedlapinequality}
    \begin{align*}
    	w^\prime (t) \geq&\; \int_{D^-} (\mathcal{L} v(x,t)) \varphi_D(x) dx + \int_{D^-} (c(x,t)-M) v(x,t) \varphi_D (x) dx \\
    	\geq& \; \int_{D^-} v(x,t) (\mathcal{L} \varphi_D(x)) dx + \int_{D^-} (c_0 (t)-M) v(x,t) \varphi_D (x) dx \\
    	=& \; -\mu_D w(t) + (c_0(t)-M)w(t).
    \end{align*}
    Now, taking
    \begin{equation*}
    	K = \sup_{t\in[0,T^\prime]} (M-c_0(t)) \geq 0,
    \end{equation*}
    we conclude
    \begin{equation*}
    	w^\prime(t) \geq (-\mu_D -K) w(t),
    \end{equation*}
    a.e. $t\in[0,T^\prime]$, which leads to
    \begin{equation*}
    	\left( e^{(\mu_D + K)t} w(t) \right)_t \geq 0.
    \end{equation*}
    Then, we have found a quantity that is nondecreasing a.e. $t\in[0,T^\prime]$ and reaches the value $0$ at $t_1$. Then, it follows that
    \begin{equation*}
    	e^{(\mu_D + K)t} w(t) \equiv 0,
    \end{equation*}
    for every $t\in[0,t_1]$, which implies
    \begin{equation*}
    	0 \equiv w(t) = \int_{D^-} v(x,t) \varphi_D(x) dx,
    \end{equation*}
    for every $t\in[0,t_1]$, where we have used the symmetry of $v$ and $\varphi_D$ for the last equality. From this and the fact that $\varphi_D(x) > 0$ for every $x\in D^-$, we conclude that $v(x,t) = 0$ for every $(x,t) \in D^- \times [0,t_1]$, and therefore
    \begin{equation*}
    	u(x,t) = 0
    \end{equation*}
    for every $(x,t) \in D^- \times [0,t_1]$.
\end{proof}

\begin{remark}
	Notice that if we consider the operator $\mathcal{L}_{a,b} = -a(-\Delta)^s + b\Delta$ with $a>0, b\geq 0$ instead of $\mathcal{L}$ in Theorem \ref{teo.movingplanes}, the proof is still valid and the result holds.
	
	In particular, the result holds if we consider just the fractional laplacian $-(-\Delta)^s$ as the operator.
\end{remark}

\begin{remark}
	Note that we could consider the bound function $C(t)$ with even less regularity than what was asked in the statement of Theorem \ref{teo.movingplanes}. If we ask for $C\in L^1_{loc}([\tau,T))$, the result is still valid but the proof gets substantially more complicated: we would need to define $v(x,t) = e^{-\int_0^t C(s) ds} u(x,t)$, which loses the continuity of its derivative $v_t$. Then we would need to work with the Dirichlet form associated to our problem \eqref{1.1}, develop a weak analogue of our reasoning and, most importantly, prove that
	\begin{equation*}
		\int_{\mathbb{R}^n} \int_{\mathbb{R}^n} \frac{(v(x,t_0)-v(y,t_0))((v^-\chi_D)(x,t_0) - (v^-\chi_D)(y,t_0))}{|x-y|^{n+2s}} dy dx \leq 0.
	\end{equation*}
	The ideas to prove this fact are the same as the ones used in our previous proof to show that $(-\Delta)^s v(x_0,t_0) \leq 0$, taking advantage of the symmetry of $u$. However, the calculations are much longer and more tedious.
\end{remark}

\section{Behaviour of the quenching phenomena} \label{sec.quenching}

First, we prove that if the parameter $\lambda$ is big enough, then every single solution of \eqref{1.1} presents quenching in finite time.

\begin{proof}[Proof of Proposition \ref{prop.todoquenchea}]
    Let $u$ be a solution of \eqref{1.1} with parameter $\lambda$ and initial datum $u_0\equiv 0$. Take $T$ as its maximal definition time. Then define the following quantity:
    \begin{equation*}
        F(t) = \int_\Omega u(x,t)\varphi_{1}(x)dx.
    \end{equation*}
Note that, by Lemma \ref{lema.velocidadpositiva}, $F(0)=0\leq F(t) \leq 1$ for every $t\in[0,T)$. Integrating by parts and using the fact that $\mathcal{L}$ is a self-adjoint operator in $L^2(\Omega)$, we get:
\begin{align*}
    F^\prime(t) =& \int_\Omega [\mathcal{L} u(x,t)] \varphi_{1}(x) dx + \lambda\int_\Omega (1-u(x,t))^{-p}\varphi_{1}(x) dx \\
    =& \int_\Omega u(x,t) [\mathcal{L} \varphi_{1}(x)] dx + \lambda\int_\Omega (1-u(x,t))^{-p}\varphi_{1}(x) dx \\
    =& -\mu_{1} F(t) + \lambda\int_\Omega (1-u(x,t))^{-p}\varphi_{1}(x) dx.
\end{align*}
Then, using Jensen's inequality, we get:
\begin{equation*}
    \lambda\int_\Omega (1-u(x,t))^{-p}\varphi_{1}(x) dx \geq \lambda \left(\int_\Omega (1-u(x,t))\varphi_{1}(x) dx \right)^{-p} \geq \lambda (1+pF(t)).
\end{equation*}
And finally, putting all together,
\begin{equation*}
    F^\prime(t) \geq \lambda + (-\mu_{1} + \lambda p) F(t) >0,
\end{equation*}
for any $t\in[0,T)$, where we have used the condition on the statement to see the positivity. Assume that $\lambda p \neq \mu_1$ (the case $\lambda p = \mu_1$ is immediate). Finally, integrating this inequality between $0$ and $T$:
\begin{equation*}
    T \leq \frac{1}{-\mu_{1}+\lambda p} \log \left( 1 + \left(\frac{-\mu_{1}}{\lambda}+p\right)\lim_{t\to T^-}F(t) \right).
\end{equation*}
The term in the right hand side is a finite positive quantity thanks to the fact that $0 \leq F(t) \leq 1$ for every $t\in[0,T)$. Then we conclude that the solution $u$ presents quenching in finite time.

Finally, take $\hat{u}$ as the solution of \eqref{1.1} with parameter $\lambda$ and initial datum $\hat{u}_0\geq 0$. Consider $\widehat{T}$ as its maximal existence time and define $T_0 = \min\{T,\widehat{T}\}$. Then define $w(x,t) = \hat{u}(x,t)- u(x,t)$ for every $(x,t) \in \Omega \times [0,T_0)$, which is clearly under the hypotheses of Lemma \ref{lema_maximumweak} for every $\tau \in (0,T_0)$. Then we conclude that $w(x,t) \geq 0$ a.e. $(x,t) \in \Omega \times (0,\tau)$. By continuity of the solutions in both space and time, we conclude that $\hat{u}(x,t) \geq u(x,t)$ for every  $(x,t)\in \Omega \times (0,T_0)$. Then $\hat{u}$ also presents quenching in finite time under the condition $\mu_{1} < \lambda (1+p)$ and the result is proved.
\end{proof}

\begin{remark}
	Notice that the principal eigenvalue $\mu_{1}$ is decreasing with respect to the size of our domain $\Omega$. This together with Proposition \ref{prop.todoquenchea} means that the bigger $\Omega$ is, the more restrictive the existence of global solutions of problem \eqref{1.1} is, as one would naturally expect.
\end{remark}

Next, let us prove that for any appropriate $u_0$, we can find a small enough $\lambda$ such that the solution of \eqref{1.1} with initial datum $u_0$ is globally defined.

To do so, we first need to study the sign of the fractional laplacian of $(1-|x|^2)_+$. Our conclusions are relayed in the following lemma, and the corresponding proof is in Appendix A.

\begin{lemma} \label{lemma.fraclaplaciansign}
	Let $\chi(x) = (1-|x|^2)_+$, with $x\in\mathbb{R}^n$. Then there exists a unique $r>0$ such that 
	\begin{itemize}
		\item $-(-\Delta)^s \chi(x) < 0$ if $|x|<r$,
		\item $-(-\Delta)^s \chi(x) = 0$ if $|x|=r$,
		\item $-(-\Delta)^s \chi(x) > 0$ if $|x|>r$.
	\end{itemize}
\end{lemma}

\begin{proof}[Proof of Proposition \ref{prop.solsglobales}]
    First, let us assume that
	\begin{equation} \label{eq_condition1}
		\lambda \leq \left( \frac{\delta}{2} \right)^{1+p} \frac{1}{A_{n,s} R^{2s}}.
	\end{equation}    
    In this case, consider the auxiliary problem
    \begin{equation}
    \left\{
    \begin{array}{ll}
        \displaystyle (-\Delta)^s \phi(x)= \lambda \left(\frac{\delta}{2}\right)^{-p}, & x\in B_R(x_0), \\
        \phi(x) = 1-\delta, & x\in\mathbb{R}^n \setminus B_R(x_0).
    \end{array}
    \right.
    \end{equation}
    It is well known that the solution to this problem is
    \begin{equation*}
        \phi(x) = (1-\delta) + A_{n,s} \lambda \left(\frac{\delta}{2}\right)^{-p}(R^2-|x-x_0|^2)^s_+.
    \end{equation*}
    This function's maximum is attained at $x=x_0$, so
    \begin{equation*}
        \max_{x\in\mathbb{R}^n} \phi(x)=\phi(x_0) = (1-\delta) + A_{n,s} \lambda \left(\frac{\delta}{2} \right)^{-p} R^{2s},
    \end{equation*}
    and using \eqref{eq_condition1}, we get that $\phi(x) \leq 1-\delta/2$ for every $x\in\Omega \subset B_R(x_0)$. What is more, since $s\in(0,1)$, we know that $\phi$ is a concave function in $B_R(x_0)$ and therefore $\Delta\phi(x) \leq 0$ for every $x\in B_R(x_0)$. Using these facts, it is easy to see that $\phi(x) \geq 1-\delta \geq 0$ for every $x\in \mathbb{R}^n$ and 
    \begin{equation*}
        -(-\Delta)^s \phi(x) + \Delta \phi(x) + \lambda (1-\phi(x))^{-p} \leq -\lambda \left(\frac{\delta}{2}\right)^{-p} + 0 + \lambda \left(\frac{\delta}{2}\right)^{-p} = 0
    \end{equation*}
    for every $x\in\Omega$. We conclude that $\phi$ is a time-independent supersolution of system \eqref{1.1} with initial datum $u_0$ and parameter $\lambda$. Then, $1-\delta/2 \geq\phi(x) \geq u(x,t)$ for every $(x,t) \in \Omega \times (0,T)$ according to Lemma \ref{lemma_comparison_particularized}.

    Next, assume that
	\begin{equation} \label{eq_condition2}
		\lambda \leq \left(\frac{\delta}{2}\right)^{1+p} \frac{2n r^2}{R^2},
	\end{equation}
    where $r$ is the one obtained in Lemma \ref{lemma.fraclaplaciansign}. For this case, consider the auxiliary problem
    \begin{equation}
    \left\{
    \begin{array}{ll}
        \displaystyle (-\Delta) \psi(x)= \lambda\left(\frac{\delta}{2}\right)^{-p}, 
        & x\in B_{R/r}(x_0),\\
        \psi(x) = 1-\delta, & x\in\mathbb{R}^n \setminus B_{R/r}(x_0).
    \end{array}
    \right.
    \end{equation}
    It is well known that the time-independent solution to this problem is
    \begin{equation*}
        \psi(x) = (1-\delta) + \frac{1}{2n} \lambda \left(\frac{\delta}{2}\right)^{-p}\left(\frac{R^2}{r^2}-|x-x_0|^2\right)_+.
    \end{equation*}
    This function's maximum is attained at $x=x_0$, so
    \begin{equation*}
        \max_{x\in\mathbb{R}^n} \psi(x)=\psi(x_0) = (1-\delta) + \frac{1}{2n} \lambda \left(\frac{\delta}{2} \right)^{-p} \frac{R^{2}}{r^2},
    \end{equation*}
    and using \eqref{eq_condition2}, we get that $\psi(x) \leq 1-\delta/2$ for every $x\in \mathbb{R}^n$. 
    
    We now want to prove that $-(-\Delta)^s\psi(x)\leq 0$ in $\Omega\subset B_R(x_0)$ to use a similar argument to the one we wrote for the first auxiliary system. Doing a simple translation and ignoring positive constants, proving $-(-\Delta)^s\psi(x)\leq 0$ in $B_{R}(x_0)$ is equivalent to proving
    \begin{equation*}
        -(-\Delta)^s \left(\frac{R^2}{r^2}-|x|^2 \right)_+ <0
    \end{equation*}
    in $B_{R}(0)$. Then we can use the scaling of the fractional laplacian to get
    \begin{equation*}
        -(-\Delta)^s \left(\frac{R^2}{r^2}-|x|^2 \right)_+ = -\frac{R^{2-2s}}{r^{2-2s}} (-\Delta)^s \left(1-\frac{|x|^2}{R^2/r^2} \right)_+.
    \end{equation*}
    Therefore, $-(-\Delta)^s\psi(x)<0$ in $B_R(x_0)$ if and only if $-[(-\Delta)^s\chi](\frac{x}{R/r})<0$ in $B_R(0)$, that is, $-[(-\Delta)^s\chi](x)<0$ in $B_r(0)$. Using Lemma \ref{lemma.fraclaplaciansign}, we conclude this is true.

    With these facts, it is easy to see that $\psi(x) \geq 1-\delta\geq 0$ for every $x\in\mathbb{R}^n$ and
    \begin{equation*}
         -(-\Delta)^s \psi(x) + \Delta \psi(x) + \lambda (1-\psi(x))^{-p} \leq 0 -\lambda \left(\frac{\delta}{2}\right)^{-p} + \lambda \left(\frac{\delta}{2}\right)^{-p} = 0
    \end{equation*}
    for every $x\in\Omega \subset B_R(x_0)$. Then $\psi$ is a time-independent supersolution of system \eqref{1.1} with initial datum $u_0$ and parameter $\lambda$. As before, according to Lemma \ref{lemma_comparison_particularized}, $1-\delta/2 \geq \psi(x) \geq u(x,t)$ for every $(x,t)\in\Omega\times(0,T).$
    
    Finally, notice that $\lambda_0=
    	\max\{\lambda_1,\lambda_2\}$, 
    where
	\[
		\lambda_1=\left(\frac\delta2\right)^{1+p}
		\frac{1}{A_{n,s}	R^{2s}} 
		\qquad\text{and}\qquad 	
		\lambda_2
		=\left(\frac\delta2\right)^{1+p}
		\frac{2nr^2}{R^2}
	\]
	are precisely the bounds appearing 
	in \eqref{eq_condition1} and \eqref{eq_condition2}. 				
	Since 
	$\lambda\le\lambda_0=\max\{\lambda_1,\lambda_2\}$, 
	we have that $\lambda\le \lambda_1$ or 
	$\lambda\le \lambda_2$, namely whichever of the 
	two realizes the maximum. 
	In the first case, \eqref{eq_condition1} holds 
	and the argument above with $\phi$ gives 
	$u(x,t)\le 1-\delta/2$ for every 
	$(x,t)\in\Omega\times(0,T)$; 
	in the second case, \eqref{eq_condition2} holds and 
	the argument with $\psi$ gives the same bound. In 
	either case, $u$ never reaches the value $1$, and 
	therefore the solution $u$ of \eqref{1.1} 
	with initial 
	datum $u_0$ and parameter $\lambda$ is global.
\end{proof}

\begin{remark}
    Note that the upper bound
    \begin{equation*}
        \lambda_0 = \left(\frac{\delta}{2}\right)^{1+p} 
        \max \left\{ \frac{1}{A_{n,s} R^{2s}}, 
        \frac{2n r^2}{R^2} \right\},
    \end{equation*}
    depends on the size of our domain $\Omega$. If $R$ is big enough, it is clear that 
    \[
    	\lambda_0 = \left(\frac{\delta}{2}\right)^{1+p} \frac{1}{A_{n,s} R^{2s}} 
    \]and therefore the fractional Laplacian is dominating over the existence of global solutions. However, if $R$ is small enough, then 
    \[
    	\lambda_0 = \left(\frac{\delta}{2}\right)^{1+p} \frac{2n r^2}{R^2}
    \] and the classical Laplacian is the one governing our result.

    This behaviour makes sense from the nature of these operators: the classical Laplacian encodes the short-range interactions of the system while the fractional Laplacian does the same with the long-range ones.
    
    Finally, notice that this result also works 
    for the purely local and nonlocal cases. 
\end{remark}

Now we aim to prove that quenching happens away from the boundary in our problem, for which we will extensively use the parabolic moving planes from Section \ref{sec.movingplanes}. This result is very important in every quenching problem, for it prevents our operator from presenting an excessively pathological behaviour: quenching at the boundary would mean that $u(x_0,t)=0$ in some $x_0\in\partial\Omega$ and for every $t\in[0,T)$, while $u(x(t),t)\to 1$ as $t\to T$ for some trajectory $x(t)\to x_0$.

\begin{proof}[Proof of Theorem \ref{teo.quenchingset}]
Consider some $u_0 \in C^{1}(\overline\Omega)$ such that the solution $u(x,t)$ of \eqref{1.1} with initial datum $u_0$ presents quenching in finite time $T$ and, by continuity,
\begin{equation} \label{cond.normalboundary}
\frac{\partial u_0}{\partial n}(x) < 0
\end{equation}
for every $x\in \Omega_\eta = \{x\in\Omega: \text{dist}(x,\mathbb{R}^n\setminus\Omega)<\eta\}$ with some small $\eta$. Take $x_0 \in \partial\Omega$, and assume without loss of generality $x_0 = 0$ with an outward normal $n = (1,0,\dots,0)$. Then take some $l \in (-\eta,0)$ and define
	\begin{equation*}
        \hat{x} = (2l-x_1,x^\prime), \;\;\; H_l = \{x=(x_1,x^\prime)\in\mathbb{R}^n : x_1 < l\}, \;\;\; \widehat{H}_l = \{x\in\mathbb{R}^n : \hat{x}\in H_l\},
    \end{equation*} 
    \begin{equation*}
        \Omega_l^+ = \Omega \cap \widehat{H}_l, \;\;\; \Omega_l^- = \{x \in \mathbb{R}^n : \hat{x} \in \Omega_l^+\}.
    \end{equation*}
    Notice that $\Omega_l^- \subset \Omega$ because $l$ can be taken as small as we need and $\Omega$ is $C^2$ and convex. Now consider the function
    \begin{equation*}
        \Phi(x,t) = u(x_1,x^\prime,t)-u(2l-x_1,x^\prime,t),
    \end{equation*}
    with $(x,t) \in H_l \times (0,T)$. It is clear that this function satisfies that $\Phi(x,t) \geq 0$ for every $(x,t) \in (H_l \setminus \Omega_l^-) \times[0,T)$, $\Phi(x,t) = - \Phi(\hat{x},t)$ for every $(x,t) \in H_l \times[0,T)$ and $\Phi(x,0) > 0$ for every $x\in\Omega_l^-$ thanks to \eqref{cond.normalboundary}. Furthermore, it is also true that
    \begin{align*}
    	\Phi_t(x,t) =& \; u_t(x_1,x^\prime,t) - u_t(2l-x_1,x^\prime,t) \\
    	=& \; \mathcal{L} \Phi(x,t) + \lambda ((1-u(x_1,x^\prime,t))^{-p} - (1-u(2l-x_1,x^\prime,t))^{-p}) \\
    	=& \;  \mathcal{L} \Phi(x,t) + \lambda p (1-\xi(x,t))^{-p-1} \Phi(x,t),
    \end{align*}
    for every $(x,t)\in\Omega_l^- \times[0,T)$, where $\xi(x,t)$ is a function that satisfies, for every $(x,t)\in\Omega_l^-\times[0,T)$:
    \begin{equation*}
    	\min\{u(x_1,x^\prime,t),u(2l-x_1,x^\prime,t)\} \leq \xi(x,t) \leq \max\{u(x_1,x^\prime,t),u(2l-x_1,x^\prime,t)\}.
    \end{equation*}
    Therefore, it is clear that, for every $(x,t)\in\Omega_l^-\times [0,T)$,
    \begin{equation*}
    	 \lambda p(1-\xi(x,t))^{-p-1} \leq \lambda p \left(1-\max_{\overline{\Omega}_l^-} \xi(\cdot,t)\right)^{-p-1} = C(t).
    \end{equation*}
   We know that $\max_{\overline{\Omega}_l^-} \xi(\cdot,t)<1$ for every $t\in[0,\tau] \subset [0,T)$, which implies that $C \in L^\infty_{loc}([0,T))$. We then conclude that, for some small fixed $l_0<0$, $\Phi$ is under the assumptions of the first part of Theorem \ref{teo.movingplanes}, where the needed regularity of $\Phi$ is given by Proposition \ref{prop.existence} and Lemma \ref{lemma.boundarycontinuity}. Then $\Phi(x,t)\geq 0$ in $\Omega_{l_0}^- \times [0,T)$. Then it is also clear that, if $x_1=l_0$:
	\begin{equation*}
		2 \frac{\partial u}{\partial x_1}(l_0,x^\prime,t) = \frac{\partial \Phi}{\partial x_1}(l_0,x^\prime,t) \leq 0,
	\end{equation*}    
     for every $t\in[0,T)$. Since this reasoning could be repeated for any $l$ with $|l|<|l_0|$, we also deduce that 
    \begin{equation*}
     	\frac{\partial u}{\partial x_1} (x,t) \leq 0, \;\; \; \text{ in } \Omega_{l_0}^+ \times (0,T).
    \end{equation*}
     
    Next, let us consider a smoothing $\widehat{\Omega}_{l_0}$ of the boundary of $\Omega_{l_0} = \Omega_{l_0}^- \cup \Omega_{l_0}^+ \cup (\Omega \cap \{x: x_1 = l_0\})$ at the points $\{x: x_1=l_0\} \cap \partial\Omega$, so that $\widehat{\Omega}_{l_0}$ has a $C^{1,1}$ boundary all over and it is still symmetric with respect to the hyperplane $\{x:x_1=l_0\}$ (see Figure \ref{fig:smoothingdomain}). Abusing slightly our notation, let us write from now on $\widehat{\Omega}_{l_0}$ as $\Omega_{l_0}$ and both its halves as $\Omega_{l_0}^-$ and $\Omega_{l_0}^+$.
    
    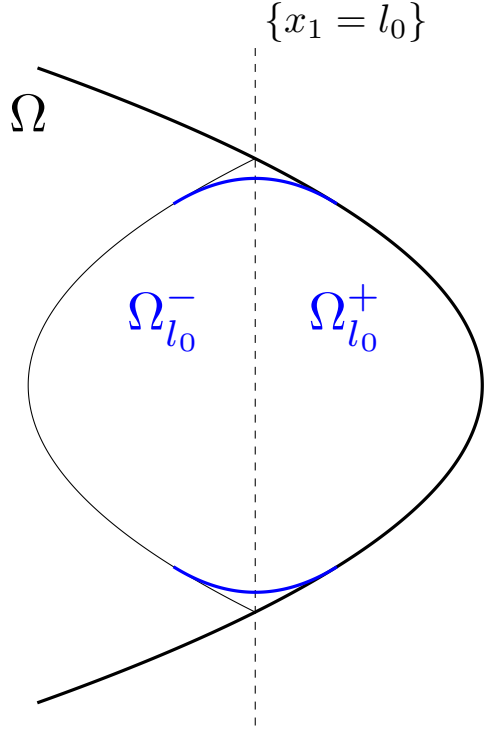
\begin{figure}[ht]
\centering

\begin{tikzpicture}[scale=3]

\draw[dashed] (0,-1.5)--(0,1.5);
\node[black, scale=1.5] at (0.4, 1.6) {$\{x_1=l_0\}$};

\draw[very thick, domain=-1.4:1.4, samples=200]
plot({1-\x*\x}, {\x});

\draw[domain=-1:1, samples=200]
plot({\x*\x-1},{\x});

\draw[blue,very thick, domain= -0.36:0.36, samples=200]
plot(\x,{0.9125-625/720*\x*\x});

\draw[blue,very thick, domain= -0.36:0.36, samples=200]
plot(\x,{-0.9125+625/720*\x*\x});

\node[black, scale=2] at (-1,1.2) {$\Omega$};
\node[blue, scale=2] at (-0.4,0.3) {$\Omega_{l_0}^-$};
\node[blue, scale=2] at (0.4,0.3) {$\Omega_{l_0}^+$};

\end{tikzpicture}

\caption{A possible smoothing of $\Omega_{l_0}$}
\label{fig:smoothingdomain}

\end{figure}
    
    Then consider the following eigenvalue problem
    \begin{equation}
    \left\{
    \begin{array}{ll}
        -\mathcal{L} v(x) = \mu v(x), & x\in \Omega_{l_0},\\
        v(x)=0, & x\in \mathbb{R}^n \setminus \Omega_{l_0},\\
        v(x) = -v(\hat{x}), & x \in \Omega_{l_0}^-.
    \end{array}
    \right.
    \end{equation}
    
    Thanks to Lemma \ref{lemma.oddeigenf}, there exists a principal eigenvalue $\mu>0$ of this problem, with an associated eigenfunction $\varphi$ that can be chosen with $\varphi < 0$ in $\Omega_{l_0}^-$ and $\|\varphi\|_{L^2(\Omega_{l_0})} = 1$.
    
    With these ingredients, define now the function $\tilde{\Phi}(x,t) = \Phi(x,t) + \varepsilon e^{-\mu t}\varphi(x)$ for every $\mathbb{R}^n \times[0,T)$. It is clear that this function satisfies that $\tilde{\Phi}(x,t) \geq 0$ for every $(x,t) \in (H_{l_0} \setminus \Omega_{l_0}^-) \times[0,T)$ and $\tilde{\Phi}(x,t) = - \tilde{\Phi}(\hat{x},t)$ for every $(x,t) \in H_{l_0} \times[0,T)$. We also claim that there exists some small $\varepsilon>0$ such that $\tilde{\Phi}(x,0) > 0$ for every $x\in\Omega_{l_0}^-$.

Indeed, since we chose $l_0$ appropriately, it is clear that $\Phi(x,0)$ is bounded away from zero for every $x\in \Omega_{l_0}^- \cap \{x: x_1\leq l_0-\delta\}$ for some $\delta>0$. We also know that $\Phi(\cdot,0) \in C^{1}(\overline{\Omega_{l_0}^-})$, $\Phi(l_0,x^\prime,0)=0$ and $\partial_{x_1} \Phi(l_0,x^\prime,0)<0$ for every $(l_0,x^\prime) \in \overline{\Omega_{l_0}^-} \cap \{x: x_1=l_0\}$. Since the set $\overline{\Omega_{l_0}^-} \cap \{x: x_1=l_0\}$ is closed, we can take the maximum of $\partial_{x_1} \Phi$ so that
\begin{equation*}
	\frac{\partial \Phi}{\partial {x_1}}(l_0,x^\prime,0) \leq \max_{x\in \overline{\Omega_{l_0}^-} \cap \{x: x_1=l_0\}} \frac{\partial \Phi}{\partial {x_1}}(x,0) < 0,
\end{equation*}
for every $(l_0,x^\prime) \in \overline{\Omega}_{l_0}^- \cap \{x: x_1=l_0\}$. This together with the fact that $\varphi \in C^{1,\beta^\prime}(\overline{\Omega}_{l_0})$, tells us that we can choose a small enough $\varepsilon>0$ so that $\partial_{x_1} \Phi(l_0,x^\prime,0) < -\varepsilon \partial_{x_1} \varphi(l_0,x^\prime,0)$ for every $(l_0,x^\prime) \in \overline{\Omega}_{l_0} \cap \{x: x_1=l_0\}$, which proves the claim.
    
Finally, we can deduce that
    \begin{align*}
    	\tilde{\Phi}_t(x,t) &- \mathcal{L} \tilde{\Phi}(x,t) - \lambda p (1-\xi(x,t))^{-p-1} \tilde{\Phi}(x,t) \\
    	=& \; \varepsilon e^{-\mu t} (- \mu \varphi - \mathcal{L} \varphi  - \lambda p (1-\xi(x,t))^{-p-1} \varphi )(x) \\
    	=& \; -\lambda p \varepsilon e^{-\mu t} (1-\xi(x,t))^{-p-1} \varphi(x) > 0,
    \end{align*}
    for every $(x,t) \in \Omega_{l_0}^- \times [0,T)$. Therefore, we are now under the hypotheses of both parts of Theorem \ref{teo.movingplanes} and we can conclude that $\tilde{\Phi}(x,t) > 0$ for every $(x,t) \in \Omega_{l_0}^- \times [0,T)$. This leads us to
    \begin{equation*}
    	u(x_1,x^\prime,t) > u(2l_0  -x_1,x^\prime,t) - \varepsilon e^{-\mu t} \varphi(x) > u(2l_0 -x_1,x^\prime,t) - \varepsilon e^{-\mu T} \varphi(x),
    \end{equation*}
    for every $(x,t) \in \Omega_{l_0}^- \times[0,T)$. Then, we can particularize this at $(x_1,x^\prime) = (3 l_0/2, 0) \in \Omega_{l_0}^-$ to get
    \begin{equation*}
    	u(l_0/2,0,t) < u(3 l_0/2,0,t) + \varepsilon e^{-\mu T} \varphi(3 l_0/2,0) < 1 - C
    \end{equation*}
	for some $C>0$ and for every $t\in(0,T)$, where we have used that $\varphi$ is strictly negative in $\Omega_{l_0}^-$. Since $\Omega$ is convex and has a smooth boundary, we can also find some small $\delta>0$ for which    
	\begin{equation*}
		u(l_0/2,x^\prime,t) < u(3 l_0/2,x^\prime,t) + \varepsilon e^{-\mu T} \varphi(3 l_0/2,x^\prime) < 1 - C^\prime,
	\end{equation*}
	for some $C^\prime>0$ and for every $x^\prime \in (-\delta,\delta)$ and $t\in(0,T)$. Finally, recall that $u_{x_1}(x,t) \leq 0$ for every $(x,t) \in \Omega_{l_0}^+ \times (0,T)$, which leads us to:
    \begin{equation*}
    	u(x_1,x^\prime,t) \leq u(l_0/2,x^\prime,t) < 1-C^\prime,
    \end{equation*}
    for every $(x_1,x^\prime) \in ([l_0/2,0] \times [-\delta,\delta]) \cap \overline\Omega$ and for every $t\in(0,T)$. Therefore, any point $x=(x_1,x^\prime) \in ((l_0/2,0] \times (-\delta,\delta)) \cap \overline\Omega$ cannot be a quenching point. Notice that this argument is valid for any boundary point $x_0 \in \partial\Omega$ and the choice of $l_0$ is independent of said point. Then, for any $x_0 \in \partial\Omega$, there exists a strip neighbourhood (with $\delta$ depending on $x_0$) in which quenching cannot happen. Choosing a finite subcovering of $\partial\Omega$, we conclude that there exists a neighbourhood of the boundary in which there can be no quenching, that is, $Q(u) \subset \Omega \setminus \Omega_{-l_0/2} = \{x\in\Omega: \text{dist}(x,\partial\Omega) \geq -l_0/2\}$.
\end{proof}

\begin{remark}
	Note that if we consider $\Omega = B_R(0)$ and radial initial data $u_0(r)$ such that $u_0^\prime(r)<0$ for every $r\in(0,R]$, then we can follow the moving plane argument as in Theorem \ref{teo.quenchingset} for any $l_0\in(0,R)$, which implies that $Q(u)=\{0\}$.
\end{remark}

\appendix

\section{Proof of Lemma \ref{lemma.fraclaplaciansign}}

    \begin{figure}[ht]
\centering

\begin{tikzpicture}[scale=2]

\draw[->, thick] (0,-2)--(0,2) node[above, scale=1] {$-(-\Delta)^s \chi (x)$};
\draw[->, thick] (0,0)--(5,0) node[right, scale=1] {$|x|$};

\draw[red, domain=0:5, samples=200]
plot({\x},{1.6*sin(deg(pi/5.5*\x-pi/2))});

\node[black, scale=1.3] at (2.75,-0.2) {$r$};
\draw[black] (2.75,0) -- ( 2.75,-0.07);

\end{tikzpicture}

\caption{Sketch of the behaviour of $-(-\Delta)^s \chi(|x|)=-(-\Delta)^s(1-|x|^2)_+$.}
\label{fig:fraclaplaciansign}

\end{figure}
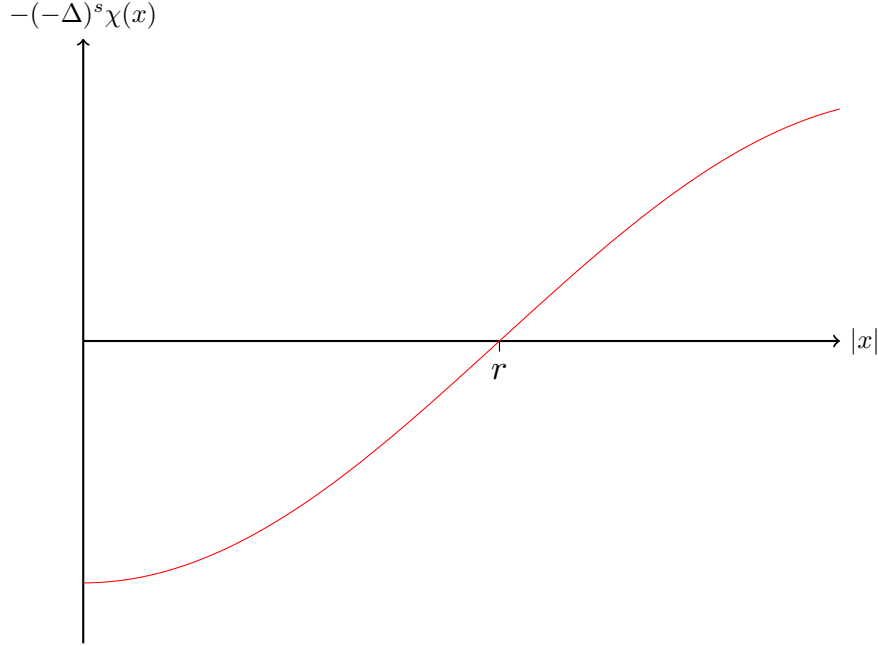

According to \cite{D2012}, we know that
    \begin{equation*}
        -(-\Delta)^s \chi(x)= -\frac{2^{2s} \Gamma(\frac{n}{2}+s) B(-s,2)}{\Gamma(-s)\Gamma(\frac{n}{2})} {}_2F_1\left(\frac{n}{2}+s, s-1;\frac{n}{2};|x|^2 \right),
    \end{equation*}
    where $\Gamma, B$ and ${}_2F_1$ are the gamma function, the beta function and the Gauss hypergeometric function respectively. It is easy to check that the fraction in this expression is negative, so the sign of $-(-\Delta)^s \chi(x)$ is the opposite of that of the Gauss hypergeometric term. 
    
    First we note that
    \begin{equation*}
        {}_2F_1\left(\frac{n}{2}+s, s-1;\frac{n}{2};0 \right) = 1,
    \end{equation*}
    and
    \begin{equation*}
        \frac{\partial}{\partial |x|}  \Big[ {}_2F_1 \Big(\frac{n}{2}+s,  s-1;\frac{n}{2};|x|^2 \Big) \Big] = 2|x| \frac{(\frac{n}{2}+s)(s-1)}{(\frac{n}{2})} {}_2F_1\left(\frac{n}{2}+s+1, s;\frac{n}{2}+1;|x|^2 \right) < 0,
    \end{equation*}
    so that the hypergeometric term is strictly decreasing with $|x|$. Finally, using properties of the hypergeometric function, we know that its limit as $|x|\to 1^-$ is negative. Indeed, if $s<1/2$,
    \begin{equation*}
        \lim_{z\to1^-} {}_2F_1\left(\frac{n}{2}+s, s-1;\frac{n}{2};z \right) = \frac{\Gamma(\frac{n}{2})\Gamma(1-2s)}{\Gamma(\frac{n}{2}-s+1)\Gamma(-s)}<0.
    \end{equation*}
    If $s=1/2$,
    \begin{equation*}
        \lim_{z\to1^-} \frac{{}_2F_1\left(\frac{n}{2}+s, s-1;\frac{n}{2};z \right)}{-\ln(1-z)} = \frac{\Gamma(\frac{n}{2})}{ \Gamma(\frac{n+1}{2})\Gamma(-\frac{1}{2})}<0.
    \end{equation*}
    Finally, if $s>1/2$,
    \begin{equation*}
        \lim_{z\to1^-} \frac{{}_2F_1\left(\frac{n}{2}+s, s-1;\frac{n}{2};z \right)}{(1-z)^{1-2s}} = \frac{\Gamma(\frac{n}{2})\Gamma(2s-1)}{ \Gamma(\frac{n}{2}+s)\Gamma(s-1)}<0.
    \end{equation*}
    Putting all together, the hypergeometric term has a unique root we will denote by $r\in(0,1)$, and the result is proved. \qed
    
    \section*{Acknowledgements}

	Both authors were partially supported by 
	ANII (Agencia Nacional de Investigaci\'on 
	e Innovaci\'on, Uruguay) 
	grant FCE\_3\_2024\_1\_181302.
	
	S. Junquera was also partially supported by grants PID2022-137074NBI00 and PID2025-168009NB-I00 funded by
MCIU/AEI/10.13039/501100011033, Spain, and by grants FPU21/05580 and EST26/00044 from the Ministry
of Science, Innovation and Universities.
	
	L.\,M.\ Del Pezzo 
	was also partially supported 
	by grant PID2023-148028NB-I00, 
	funded by \break MICIU/
	AEI/10.13039/501100011033 and by ERDF/EU.

\bibliographystyle{abbrvdoi}
\bibliography{mixed_operator_bibliography}

\end{document}